\documentclass[12pt,reqno]{amsart}
\usepackage{eurosym}
\usepackage{amssymb}
\usepackage{amscd}
\usepackage{amsmath}
\usepackage{amsfonts}
\usepackage{color}
\usepackage{caption}
\usepackage{wrapfig}
\usepackage{hyperref}
\usepackage{tikz}

\theoremstyle{plain}
\newtheorem*{acknowledgment}{Acknowledgment}

\newtheorem{algorithm}{Algorithm}[section]

\newtheorem{corollary}[algorithm]{Corollary}

\newtheorem{definition}[algorithm]{Definition}

\newtheorem{example}[algorithm]{Example}

\newtheorem{addendum}[algorithm]{Addendum}
\newtheorem{lemma}[algorithm]{Lemma}

\newtheorem{theorem} [algorithm] {Theorem}
\newtheorem{keylemma}[algorithm]{Key Lemma}

\newtheorem{proposition}[algorithm]{Proposition}
\newtheorem{remark}[algorithm]{Remark}

\numberwithin{equation}{algorithm}
\usetikzlibrary{calc,fadings,decorations.pathreplacing,patterns,matrix,arrows}

\tikzset{  >=latex,   inner sep=0pt,  outer sep=2pt,  mark coordinate/.style={inner sep=0pt,outer sep=0pt,minimum size=3pt,
    fill=black,circle}}

\begin{document}
\title{Crosscap Stability}
\author{Curtis Pro}
\address{{Department of Mathematics, University of Notre Dame} }
\email{cpro@nd.edu}
\urladdr{www.nd.edu/$\sim$cpro}
\author{Michael Sill}
\address{{Department of Mathematics, Calbaptist, Riverside} }
\email{msill@calbaptist.edu}
\urladdr{}
\author{Frederick Wilhelm}
\address{{Department of Mathematics, University of California, Riverside}}
\email{fred@math.ucr.edu}
\urladdr{https://sites.google.com/site/frederickhwilhelmjr/home}

\begin{abstract}
We provide an alternative proof that Crosscaps are diffeomorphically stable.
\end{abstract}

\maketitle

\section{Introduction}

Perelman's celebrated Stability Theorem in particular implies the following.

\bigskip

\noindent \textbf{Topological Stability Theorem:} \emph{Let }$\{M_{\alpha
}\}_{\alpha }$\emph{\ be a sequence of closed Riemannian }$n$\emph{%
--manifolds with sectional curvature }$\geq k.$\emph{\ If the
Gromov-Hausdorff limit of }$\{M_{\alpha }\}_{\alpha }$\emph{\ is }$X$\emph{\
and }$\dim \left( X\right) =n,$\emph{\ then all but finitely many of the }$%
M_{\alpha }$\emph{'s are homeomorphic to }$X.$ \cite{Kap}, \cite{Perel}

\bigskip

Given this, it is at least natural to ask the

\medskip

\noindent \textbf{Diffeomorphism Stability Question.} \emph{Let }$%
\{M_{\alpha }\}_{\alpha }$\emph{\ be a sequence of closed Riemannian }$n$%
\emph{--manifolds with sectional curvature }$\geq k.$\emph{\ If the
Gromov-Hausdorff limit of }$\{M_{\alpha }\}_{\alpha }$\emph{\ is }$X$\emph{\
and }$\dim \left( X\right) =n,$\emph{\ then are all but finitely many of the 
}$M_{\alpha }$\emph{'s diffeomorphic to each other? }\cite{GrovWilh2}

\bigskip

An affirmative answer to the Diffeomorphism Stability Theorem would provide
a generalizations of Cheeger's Finiteness Theorem and the Diameter Sphere
Theorem \cite{Cheeg1}, \cite{Cheeg2}, \cite{GrovShio}, \cite{GrovWilh2}.

\begin{definition}
Let $\mathcal{M}_{k}\left( n\right) $ be the class of closed Riemannian $n$%
--manifolds with sectional curvature $\geq k.$ A compact, $n$--dimensional $%
X\in \mathrm{closure}\left( \mathcal{M}_{k}\left( n\right) \right) $ is
called diffeomorphically stable if for any sequence $\left\{ M_{\alpha
}\right\} _{\alpha =1}^{\infty }\subset \mathcal{M}_{k}\left( n\right) $
with $M_{\alpha }\longrightarrow X,$ in the Gromov--Hausdorff topology, all
but finitely many of the $M_{\alpha }$s are diffeomorphic to each other.
\end{definition}

It follows from Theorem 6.1 in \cite{KMS} that $X$ is diffeomorphically
stable if all of its points are $\left( n,\delta \right) $--strained in the
sense of \cite{BGP}. Examples of such spaces are curvature $k$ crosscaps.

\begin{example}
\label{crosscapexample} \textbf{(Crosscap) }Let\textbf{\ }$\mathcal{D}%
_{k}^{n}\left( r\right) $ be an $r$--ball in the $n$--dimensional, simply
connected space form of constant curvature $k.$\textbf{\ }

The constant curvature $k$ Crosscap, $C_{k,r}^{n},$ is the quotient of $%
\mathcal{D}_{k}^{n}\left( r\right) $ obtained by identifying antipodal
points on the boundary. Thus $C_{k,r}^{n}$ is homeomorphic to $\mathbb{R}%
P^{n}$. There is a canonical metric on $C_{k,r}^{n}$ that makes this
quotient map a submetry. The universal cover of $C_{k,r}^{n}$ is the double
of $\mathcal{D}_{k}^{n}\left( r\right) $. If we write this double as $%
\mathbb{D}_{k}^{n}\left( r\right) :=\mathcal{D}_{k}^{n}\left( r\right)
^{+}\cup _{\partial \mathcal{D}_{k}^{n}(r)^{\pm }}\mathcal{D}_{k}^{n}\left(
r\right) ^{-},$ then the free involution%
\begin{equation*}
A:\mathbb{D}_{k}^{n}\left( r\right) \longrightarrow \mathbb{D}_{k}^{n}\left(
r\right)
\end{equation*}%
that gives the covering map $\mathbb{D}_{k}^{n}\left( r\right)
\longrightarrow C_{k,r}^{n}$ is%
\begin{equation*}
A:\left( x,+\right) \longmapsto \left( -x,-\right) ,
\end{equation*}%
where the sign in the second entry indicates whether the point is in $%
\mathcal{D}_{k}^{n}(r)^{+}$ or $\mathcal{D}_{k}^{n}(r)^{-}.$
\end{example}

Crosscaps are of particular interest since they are one of the three types
of limit spaces with maximal volume in the sense of \cite{GrovPet3}. Here we
provide an alternative proof that Crosscaps are diffeomorphically stable.

\begin{theorem}
\label{Cross Cap Stability}All Crosscaps are diffeomorphically stable.

In other words, let $\left\{ M_{i}\right\} _{i=1}^{\infty }$ be a sequence
of closed Riemannian $n$--manifolds with $\mathrm{sec}$ $M_{i}\geq k$ so
that 
\begin{equation*}
M_{i}\longrightarrow C_{k,r}^{n}
\end{equation*}%
in the Gromov-Hausdorff topology. Then all but finitely many of the $M_{i}$s
are diffeomorphic to $\mathbb{R}P^{n}.$
\end{theorem}

Since all points of Crosscaps are $\left( n,0\right) $--strained, this
follows from Theorem 6.1 in \cite{KMS}. We present a different proof here
that is of independent interest because of its relationship to the proofs of
the main theorems in \cite{OSY} and \cite{ProSillWilh}.

\begin{remark}
Theorem \ref{Cross Cap Stability} when $k=1$ and $r=\frac{\pi }{2}$ follows
from the main theorem in \cite{Yam1} and the fact that $C_{1,\frac{\pi }{2}%
}^{n}$ is $\mathbb{R}P^{n}$ with constant curvature $1$ $.$
\end{remark}

Section 2 introduces notations and conventions. Section 3 is review of
necessary tools from Alexandrov geometry. Section 4 develops machinery and
proves Theorem \ref{Cross Cap Stability} in the case when $n\neq 4.$ Theorem %
\ref{Cross Cap Stability} in dimension 4 is proven in Section 5.

Throughout the remainder of the paper, we assume without loss of generality,
by rescaling if necessary, that $k=-1,0$ or $1$.

\begin{acknowledgment}
We are grateful to Stefano Vidussi for several conversations about exotic
differential structures on $\mathbb{R}P^{4}.$

We are grateful to a referee of this paper for making us aware of the
results in \cite{KMS}.
\end{acknowledgment}

\section{Conventions and Notations}

We assume a basic familiarity with Alexandrov spaces, including but not
limited to \cite{BGP}. Let $X$ be an $n$--dimensional Alexandrov space and $%
x,p,y\in X$.

\begin{enumerate}
\item We call minimal geodesics in $X$ \emph{segments}. We denote by $px$ a
segment in $X$ with endpoints $p$ and $x$.

\item We let $\Sigma_{p}$ and $T_{p}X$ denote the space of directions and
tangent cone at $p$, respectively.

\item For $v\in T_{p}X$ we let $\gamma _{v}$ be the segment whose initial
direction is $v.$

\item Following \cite{Pet}, $\Uparrow _{x}^{p}\subset \Sigma _{x}$ will
denote the set of directions of segments from $x$ to $p,$ and $\uparrow
_{x}^{p}\in $ $\Uparrow _{x}^{p}$ denotes the direction of a single segment
from $x$ to $p.$

\item We let $\sphericalangle (x,p,y)$ denote the angle of a hinge formed by 
$px$ and $py$ and $\tilde{\sphericalangle}(x,p,y)$ denote the corresponding
comparison angle.

\item Following \cite{OSY}, we let $\tau :\mathbb{R}^{k}\rightarrow \mathbb{R%
}_{+}$ be any function that satisfies 
\begin{equation*}
\lim_{x_{1},\ldots ,x_{k}\rightarrow 0}\tau \left( x_{1},\ldots
,x_{k}\right) =0,
\end{equation*}%
and abusing notation we let $\tau :\mathbb{R}^{k}\times \mathbb{R}%
^{n}\rightarrow \mathbb{R}$ be any function that satisfies 
\begin{equation*}
\lim_{x_{1},\ldots ,x_{k}\rightarrow 0}\tau \left( x_{1},\ldots
,x_{k}|y_{1},\ldots ,y_{n}\right) =0,
\end{equation*}%
provided that $y_{1},\ldots ,y_{n}$ remain fixed.

When making an estimate with a function $\tau $ we implicitly assert the
existence of such a function for which the estimate holds.

\item We denote by $\mathbb{R}^{1,n}$ the Minkowski space $(\mathbb{R}%
^{n+1},g),$ where $g$ is the semi-Riemannian metric defined by 
\begin{equation*}
g=-dx_{0}^{2}+dx_{1}^{2}+\cdots +dx_{n}^{2}
\end{equation*}%
for coordinates $(x_{0},x_{1},\cdots ,x_{n})$ on $\mathbb{R}^{n+1}$.

\item We reserve $\{e_{j}\}_{j=0}^{m}$ for the standard orthonormal basis in
both euclidean and Minkowski space.

\item We use two isometric models for hyperbolic space, 
\begin{equation*}
H_{+}^{n}:=\left\{ \left. (x_{0},x_{1},\cdots ,x_{n})\in \mathbb{R}%
^{n+1}\right\vert -\left( x_{0}\right) ^{2}+\left( x_{1}\right) ^{2}+\cdots
+\left( x_{n}\right) ^{2}=-1,\text{ }x_{0}>0\right\}
\end{equation*}%
and 
\begin{equation*}
H_{-}^{n}:=\left\{ \left. (x_{0},x_{1},\cdots ,x_{n})\in \mathbb{R}%
^{n+1}\right\vert -\left( x_{0}\right) ^{2}+\left( x_{1}\right) ^{2}+\cdots
+\left( x_{n}\right) ^{2}=-1,\text{ }x_{0}<0\right\} .
\end{equation*}

\item We obtain explicit double disks, $\mathbb{D}_{k}^{n}\left( r\right) :=%
\mathcal{D}_{k}^{n}\left( r\right) ^{+}\cup _{\partial \mathcal{D}%
_{k}^{n}(r)^{\pm }}\mathcal{D}_{k}^{n}\left( r\right) ^{-},$ by viewing $%
\mathcal{D}_{k}^{n}\left( r\right) ^{+}$ and $\mathcal{D}_{k}^{n}\left(
r\right) ^{-}$ explicitly as 
\begin{equation*}
\mathcal{D}_{k}^{n}\left( r\right) ^{+}:=\left[ 
\begin{array}{cc}
\left\{ \left. z\in H_{+}^{n}\subset \mathbb{R}^{1,n}\right\vert \mathrm{dist%
}_{H_{+}^{n}}\left( e_{0},z\right) \leq r\right\} & \text{if }k=-1 \\ 
\left\{ \left. z\in \left\{ e_{0}\right\} \times \mathbb{R}^{n}\subset 
\mathbb{R}^{n+1}\right\vert \mathrm{dist}_{\mathbb{R}^{n+1}}\left(
e_{0},z\right) \leq r\right\} & \text{if }k=0 \\ 
\left\{ \left. z\in S^{n}\subset \mathbb{R}^{n+1}\right\vert \mathrm{dist}%
_{S^{n}}\left( e_{0},z\right) \leq r\right\} & \text{if }k=1,%
\end{array}%
\right.
\end{equation*}%
and%
\begin{equation*}
\mathcal{D}_{k}^{n}\left( r\right) ^{-}:=\left[ 
\begin{array}{cc}
\left\{ \left. z\in H_{-}^{n}\subset \mathbb{R}^{1,n}\right\vert \mathrm{dist%
}_{H_{-}^{n}}\left( -e_{0},z\right) \leq r\right\} & \text{if }k=-1 \\ 
\left\{ \left. z\in \left\{ -e_{0}\right\} \times \mathbb{R}^{n}\subset 
\mathbb{R}^{n+1}\right\vert \mathrm{dist}_{\mathbb{R}^{n+1}}\left(
-e_{0},z\right) \leq r\right\} & \text{if }k=0 \\ 
\left\{ \left. z\in S^{n}\subset \mathbb{R}^{n+1}\right\vert \mathrm{dist}%
_{S^{n}}\left( -e_{0},z\right) \leq r\right\} & \text{if }k=1.%
\end{array}%
\right.
\end{equation*}
\end{enumerate}

Since $r<\frac{\pi }{2}$ when $k=1,$ $\mathcal{D}_{k}^{n}\left( r\right)
^{+} $ and $\mathcal{D}_{k}^{n}\left( r\right) ^{-}$ are disjoint in all
three cases.

%
%
%

\section{Basic Tools From Alexandrov Geometry}

The notion of strainers \cite{BGP} in an Alexandrov space forms the core of
the calculus arguments used to prove our main theorem. In this section, we
review this notion and its relevant consequences. In some sense the idea can
be traced back to \cite{OSY}, and some of the ideas that we review first
appeared in other sources such as \cite{Wilh} and \cite{Yam2}.

\begin{definition}
Let $X$ be an Alexandrov space. A point $x\in X$ is said to be $\left(
n,\delta ,r\right) $--strained by the strainer $\left\{ \left(
a_{i},b_{i}\right) \right\} _{i=1}^{n}\subset X\times X$ provided that for
all $i\neq j$ we have%
\begin{equation*}
\begin{array}{ll}
\widetilde{\sphericalangle }\left( a_{i},x,b_{j}\right) >\frac{\pi }{2}%
-\delta , & \widetilde{\sphericalangle }\left( a_{i},x,b_{i}\right) >\pi
-\delta , \\ 
\widetilde{\sphericalangle }\left( a_{i},x,a_{j}\right) >\frac{\pi }{2}%
-\delta , & \widetilde{\sphericalangle }\left( b_{i},x,b_{j}\right) >\frac{%
\pi }{2}-\delta ,\text{ and} \\ 
\multicolumn{2}{c}{\min_{i=1,\ldots ,n}\left\{ \mathrm{dist}%
(\{a_{i},b_{i}\},x)\right\} >r.}%
\end{array}%
\end{equation*}

We say a metric ball $B\subset X$ is an $(n,\delta ,r)$--strained
neighborhood with strainer $\{a_{i},b_{i}\}_{i=1}^{n}$ provided every point $%
x\in B$ is $(n,\delta ,r)$--strained by $\{a_{i},b_{i}\}_{i=1}^{n}$.
\end{definition}

The following is observed in \cite{Yam2}.

\begin{proposition}
\label{posdeltastrainedrad} Let $X$ be a compact $n$-dimensional Alexandrov
space. Then the following are equivalent.

\begin{description}
\item[1] There is a (sufficiently small) $\eta >0$ so that for every $p\in X$%
\begin{equation*}
\mathrm{dist}_{G-H}\left( \Sigma _{p},S^{n-1}\right) <\eta .
\end{equation*}

\item[2] There is a (sufficiently small) $\delta >0$ and an $r>0$ such that $%
X$ is covered by finitely many $(n,\delta ,r)$--strained neighborhoods.
\end{description}
\end{proposition}

\begin{theorem}
\label{BGP coordinates}(\cite{BGP} Theorem 9.4) Let $X$ be an $n$%
--dimensional Alexandrov space with curvature bounded from below. Let $p\in
X $ be $\left( n,\delta ,r\right) $--strained by $\left\{ \left(
a_{i},b_{i}\right) \right\} _{i=1}^{n}.$ Provided $\delta $ is small enough,
there is a $\rho >0$ such that the map $f:B(p,\rho )\rightarrow \mathbb{R}%
^{n}$ defined by 
\begin{equation*}
f(x)=\left( \mathrm{dist}\left( a_{1},x\right) ,\mathrm{dist}\left(
a_{2},x\right) ,\ldots ,\mathrm{dist}\left( a_{n},x\right) \right)
\end{equation*}%
is a bi-Lipschitz embedding with Lipschitz constants in $\left( 1-\tau
\left( \delta ,\rho \right) ,1+\tau \left( \delta ,\rho \right) \right) .$
\end{theorem}

If every point in $X$ is $(n,\delta ,r)$--strained, we can equip $X$ with a $%
C^{1}$--differentiable structure defined by Otsu and Shioya in \cite{OS}.
The charts will be smoothings of the map from the theorem above and are
defined as follows: Let $x\in X$ and choose $\sigma >0$ so that $B(x,\sigma
) $ is $(n,\delta ,r)$--strained by $\{a_{i},b_{i}\}_{i=1}^{n}$. Define $%
d_{i,x}^{\eta }:B(x,\sigma )\rightarrow \mathbb{R}$ by 
\begin{equation*}
d_{i,x}^{\eta }(y)=\frac{1}{\mathrm{vol}(B(a_{i},\eta ))}\int_{z\in
B(a_{i},\eta )}\mathrm{dist}(y,z).
\end{equation*}%
Then $\varphi _{x}^{\eta }:B(x,\sigma )\rightarrow \mathbb{R}^{n}$ is
defined by 
\begin{equation}
\varphi _{x}^{\eta }(y)=(d_{1,x}^{\eta }(y),\ldots ,d_{n,x}^{\eta }(y)).
\label{Otsu-Shioya}
\end{equation}

If $B$ is $(n,\delta ,r)$--strained by $\{a_{i},b_{i}\}_{i=1}^{n}$, any
choice of $2n$--directions $\left\{ \left( \uparrow _{x}^{a_{i}},\uparrow
_{x}^{b_{i}}\right) \right\} _{i=1}^{n}$ where $x\in B$ will be called a set
of straining directions for $\Sigma _{x}.$ As in, \cite{BGP,Yam2}, we say an
Alexandrov space $\Sigma $ with $\mathrm{curv\,}\Sigma \geq 1$ is globally $%
(m,\delta )$-strained by pairs of subsets $\{A_{i},B_{i}\}_{i=1}^{m}$
provided 
\begin{equation*}
\begin{array}{ll}
|\mathrm{dist}(a_{i},b_{j})-\frac{\pi }{2}|<\delta , & \mathrm{dist}%
(a_{i},b_{i})>\pi -\delta , \\ 
|\mathrm{dist}(a_{i},a_{j})-\frac{\pi }{2}|<\delta , & |\mathrm{dist}%
(b_{i},b_{j})-\frac{\pi }{2}|<\delta%
\end{array}%
\end{equation*}%
for all $a_{i}\in A_{i}$, $b_{i}\in B_{i}$ and $i\neq j$.

\begin{theorem}
\label{BGP--SOY}(\cite{BGP} Theorem 9.5, cf also \cite{OSY} Section 3) Let $%
\Sigma $ be an $\left( n-1\right) $--dimensional Alexandrov space with
curvature $\geq 1.$ Suppose $\Sigma $ is globally strained by $%
\{A_{i},B_{i}\}$. There is a map $\tilde{\Psi}:\mathbb{R}^{n}\longrightarrow
S^{n-1}$ so that $\Psi :\Sigma \rightarrow S^{n-1}$ defined by 
\begin{equation*}
\Psi (x)=\tilde{\Psi}\circ \left( \mathrm{dist}\left( A_{1},x\right) ,%
\mathrm{dist}\left( A_{2},x\right) ,\ldots ,\mathrm{dist}\left(
A_{n},x\right) \right)
\end{equation*}%
is a bi-Lipschitz homeomorphisms with Lipshitz constants in $\left( 1-\tau
\left( \delta \right) ,1+\tau \left( \delta \right) \right) $.
\end{theorem}

\begin{remark}
The description of $\tilde{\Psi}:\mathbb{R}^{n}\longrightarrow S^{n-1}$ in 
\cite{BGP} is explicit but is geometric rather than via a formula. Combining
the proof in \cite{BGP} with a limiting argument, one can see that the map $%
\Psi $ can be given by 
\begin{equation*}
\Psi (x)=\left( \sum \cos ^{2}\left( \mathrm{dist}\left( A_{i},x\right)
\right) \right) ^{-1/2}\left( \cos \left( \mathrm{dist}\left( A_{1},x\right)
\right) ,\ldots ,\cos \left( \mathrm{dist}\left( A_{n},x\right) \right)
\right) .
\end{equation*}%
In particular, the differentials of $\varphi _{x}^{\eta }:B(x,\sigma
)\subset X\longrightarrow \varphi (B(x,\sigma ))$ are almost isometries.
\end{remark}

Next we state a powerful lemma showing that for an $(n,\delta ,r)$ strained
neighborhood, angle and comparison angle almost coincide for geodesic hinges
with one side in this neighborhood and the other reaching a strainer.

\begin{lemma}
(\cite{BGP} Lemma $5.6$) Let $B\subset X$ be $\left( 1,\delta ,r\right) $%
--strained by $(y_{1},y_{2}).$ For any $x,z\in B$%
\begin{equation*}
\left\vert \tilde{\sphericalangle}\left( y_{1},x,z\right) +\tilde{%
\sphericalangle}\left( y_{2},x,z\right) -\pi \right\vert <\tau \left( \delta
,\mathrm{dist}\left( x,z\right) |r\right)
\end{equation*}%
In particular, for $i=1,2$, 
\begin{equation*}
\left\vert \sphericalangle \left( y_{i},x,z\right) -\tilde{\sphericalangle}%
\left( y_{i},x,z\right) \right\vert <\tau \left( \delta ,\mathrm{dist}\left(
x,z\right) |r\right) .
\end{equation*}
\end{lemma}

\begin{corollary}
\label{Angle continuity}Let $B\subset X$ be $\left( 1,\delta ,r\right) $%
--strained by $\left( a,b\right) $. Let $\left\{ X^{\alpha }\right\}
_{\alpha =1}^{\infty }$ be a sequence of Alexandrov spaces with $\mathrm{curv%
}X^{\alpha }\geq k$ such that $X^{\alpha }\longrightarrow X.$ For $x,z\in B$%
, suppose that $a^{\alpha },b^{\alpha },x^{\alpha },z^{\alpha }\in X^{\alpha
}$ converge to $a,b,x,$ and $z$ respectively. Then 
\begin{equation*}
\left\vert \sphericalangle \left( a^{\alpha },x^{\alpha },z^{\alpha }\right)
-\sphericalangle \left( a,x,z\right) \right\vert <\tau \left( \delta ,%
\mathrm{dist}\left( x,z\right) ,\tau \left( 1/\alpha |\mathrm{dist}\left(
x,z\right) \right) \text{ }|\text{ }r\right) .
\end{equation*}
\end{corollary}

\begin{proof}
The convergence $X^{\alpha }\longrightarrow X$ implies that we have
convergence of the corresponding comparison angles. The result follows from
the previous lemma.
\end{proof}

\begin{lemma}
\label{angle convergence}Let $B\subset X$ be $(n,\delta ,r)$--strained by $%
\left\{ \left( a_{i},b_{i}\right) \right\} _{i=1}^{n}$. Let $\left\{
X^{\alpha }\right\} _{\alpha =1}^{\infty }$ have $\mathrm{curv}X^{\alpha
}\geq k$ and suppose that $X_{\alpha }\longrightarrow $ $X$. Let $\left\{
\left( \gamma _{1,\alpha },\gamma _{2,\alpha }\right) \right\} _{\alpha
=1}^{\infty }$ be a sequence of geodesic hinges in the $X^{\alpha }$ that
converge to a geodesic hinge $\left( \gamma _{1},\gamma _{2}\right) $ with
vertex in $B.$ Then 
\begin{equation*}
\left\vert \sphericalangle \left( \gamma _{1,\alpha }^{\prime }\left(
0\right) ,\gamma _{2,\alpha }^{\prime }\left( 0\right) \right)
-\sphericalangle \left( \gamma _{1}^{\prime }\left( 0\right) ,\gamma
_{2}^{\prime }\left( 0\right) \right) \right\vert <\tau \left( \delta ,\tau
\left( 1/\alpha |\mathrm{len}\left( \gamma _{1}\right) ,\mathrm{len}\left(
\gamma _{2}\right) \right) \text{ }|\text{ }r\right) .
\end{equation*}
\end{lemma}

\begin{remark}
Note that without the strainer, $\lim \inf_{\alpha \rightarrow \infty
}\sphericalangle \left( \gamma _{1,\alpha }^{\prime }\left( 0\right) ,\gamma
_{2,\alpha }^{\prime }\left( 0\right) \right) \geq \sphericalangle \left(
\gamma_1 ^{\prime }\left( 0\right) ,\gamma_2 ^{\prime }\left( 0\right)
\right) $ \cite{GrovPet2}, \cite{BGP}.
\end{remark}

\begin{proof}
Apply the previous corollary with $x^{\alpha }=\gamma _{1,\alpha }\left(
0\right) ,$ $z^{\alpha }=\gamma _{1,\alpha }\left( \varepsilon \right) ,$ $%
x^{\alpha}\rightarrow x,$ and $z^{\alpha}\rightarrow z$ to conclude 
\begin{equation*}
\left\vert \sphericalangle (\Uparrow _{x^{\alpha }}^{a_{i}^{\alpha }},\gamma
_{1,\alpha }^{\prime }\left( 0\right) )-\sphericalangle (\Uparrow
_{x}^{a_{i}},\gamma _{1}^{\prime }\left( 0\right) )\right\vert <\tau \left(
\delta ,\mathrm{dist}\left( x,z\right) ,\tau \left( 1/\alpha |\mathrm{dist}%
\left( x,z\right) \right) \text{ }|\text{ }r\right) .
\end{equation*}%
Similar reasoning with $x^{\alpha }=\gamma _{2,\alpha }\left( 0\right) ,$ $%
z^{\alpha }=\gamma _{2,\alpha }\left( \varepsilon \right) ,$ $x=\lim_{\alpha
\rightarrow \infty }x^{\alpha },$ and $z=\lim_{\alpha \rightarrow \infty
}z^{\alpha }$ gives 
\begin{equation*}
\left\vert \sphericalangle (\Uparrow _{x^{\alpha }}^{a_{i}^{\alpha }},\gamma
_{2,\alpha }^{\prime }\left( 0\right) )-\sphericalangle (\Uparrow
_{x}^{a_{i}},\gamma _{2}^{\prime }\left( 0\right) )\right\vert <\tau \left(
\delta ,\mathrm{dist}\left( x,z\right) ,\tau \left( 1/\alpha |\mathrm{dist}%
\left( x,z\right) \right) \text{ }|\text{ }r\right) .
\end{equation*}

Since $\mathrm{dist}\left( x,z\right) $ may be as small as we please, the
result then follows from Theorem \ref{BGP--SOY}.
\end{proof}

\begin{lemma}
\label{Dense Injectivity}(\cite{Yam2} Lemma 1.8.2) Let $\left\{ \left(
a_{i},b_{i}\right) \right\} _{i=1}^{n}$ be an $\left( n,\delta ,r\right) $%
--strainer for $B\subset X.$ For any $x\in B$ and $\mu >0,$ let $\Sigma
_{x}^{\mu }$ be the set of directions $v\in \Sigma _{x}$ so that $\gamma
_{v}|_{\left[ 0,\mu \right] }$ is a segment. For any sufficiently small $\mu
>0,$ $\Sigma _{x}^{\mu }$ is $\tau \left( \delta ,\mu \right) $--dense in $%
\Sigma _{x}.$
\end{lemma}

\begin{corollary}
\label{convergence of vectors}Suppose $X^{\alpha }\longrightarrow X$, $%
\left\{ \left( a_{i},b_{i}\right) \right\} _{i=1}^{n}$ is an $\left(
n,\delta ,r\right) $--strainer for $B\subset X,$ and $\left( n,\delta
,r\right) $--strainers $\left\{ \left( a_{i}^{\alpha },b_{i}^{\alpha
}\right) \right\} _{i=1}^{n}$ for $B^{\alpha }\subset X^{\alpha }$ satisfy 
\begin{equation*}
\left( \left\{ \left( a_{i}^{\alpha },b_{i}^{\alpha }\right) \right\}
_{i=1}^{n},B^{\alpha }\right) \longrightarrow \left( \left\{ \left(
a_{i},b_{i}\right) \right\} _{i=1}^{n},B\right) .
\end{equation*}%
For any fixed $\mu >0$ and any sequence of directions $\left\{ v^{\alpha
}\right\} _{a=1}^{\infty }\subset \Sigma _{x^{\alpha }}$ with $x^{\alpha
}\in B^{\alpha },$ there is a sequence $\left\{ w^{\alpha }\right\}
_{a=1}^{\infty }\subset \Sigma _{x^{\alpha }}^{\mu }$ with 
\begin{equation*}
\sphericalangle \left( w^{\alpha },v^{\alpha }\right) <\tau \left( \delta
,\mu \right)
\end{equation*}%
so that a subsequence of $\left\{ \gamma _{w^{\alpha }}\right\} _{\alpha
=1}^{\infty }$ converges to a geodesic $\gamma :\left[ 0,\mu \right]
\longrightarrow X.$
\end{corollary}

From Arzela-Ascoli and Hopf-Rinow, we conclude

\begin{proposition}
\label{geodesic continuity}Let $X$ be an Alexandrov space and $p,q\in X.$
For any $\varepsilon >0,$ there is a $\delta >0$ so that for all $x\in
B\left( p,\delta \right) $ and all $y\in B\left( q,\delta \right) $ and any
segment $xy,$ there is a segment $pq$ so that 
\begin{equation*}
\mathrm{dist}\left( xy,pq\right) <\varepsilon .
\end{equation*}
\end{proposition}

We end this section by showing that convergence to a compact Alexandrov
space $X$ without collapse implies the convergence of the corresponding
universal covers, provided $\left\vert \pi _{1}\left( X\right) \right\vert
<\infty .$ For our purposes, when $X=C_{k,r}^{n}$, it would be enough to use 
\cite{SorWei} or \cite{FukYam}.

The key tools are Perelman's Stability and Local Structure Theorems and the
notion of first systole, which is the length of the shortest closed
non-contractible curve. Perelman's proof of the Local Structure Theorem can
be found in \cite{Perel}, this result is also a corollary to his Stability
Theorem, whose proof is published in \cite{Kap}.

\begin{theorem}
\label{covers converge}Let $\left\{ X_{i}\right\} _{i=1}^{\infty }$ be a
sequence of $n$--dimensional Alexandrov spaces with a uniform lower
curvature bound converging to a compact, $n$--dimensional Alexandrov space $%
X.$ If the fundamental group of $X$ is finite, then

\begin{description}
\item[1] A subsequence of the universal covers, $\{\tilde{X}%
_{i}\}_{i=1}^{\infty },$ of $\left\{ X_{i}\right\} _{i=1}^{\infty }$%
converges to the universal cover, $\tilde{X},$ of $X$.

\item[2] A subsequence of the deck action by $\pi _{1}\left( X_{i}\right) $
on $\{ \tilde{X}_{i}\} _{i=1}^{\infty }$ converges to the deck action of $%
\pi _{1}\left( X\right) $ on $\tilde{X}.$
\end{description}
\end{theorem}

\begin{proof}
In \cite{Perel}, Perelman shows $X$ is locally contractible. Let $\left\{
U_{j}\right\} _{j=1}^{n}$ be an open cover of $X$ by contractible sets and
let $\mu $ be a Lebesgue number of this cover. By Perelman's Stability
Theorem, there are $\tau \left( \frac{1}{i}\right) $--Hausdorff
approximations 
\begin{equation*}
h_{i}:X\longrightarrow X_{i}
\end{equation*}%
that are also homeomorphisms. Therefore, if $i$ is sufficiently large, $%
\left\{ h_{i}\left( U_{j}\right) \right\} _{j=1}^{n}$ is an open cover for $%
X_{i}$ by contractible sets with Lebesgue number $\mu /2$. It follows that
the first systoles of the $X_{i}$s are uniformly bounded from below by $\mu
. $ Since the minimal displacement of the deck transformations by $\pi
_{1}\left( X_{i}\right) $ on $\tilde{X}_{i}\longrightarrow X_{i}$ is equal
to the first systole of $X_{i}$, this displacement is also uniformly bounded
from below by $\mu .$ By precompactness, a subsequence of $\{\tilde{X}_{i}\}$
converges to a length space $Y.$ From Proposition 3.6 of \cite{FukYam}, a
subsequence of the actions $\left( \tilde{X}_{i},\pi _{1}\left( X_{i}\right)
\right) $ converges to an isometric action by some group $G$ on $Y.$ By
Theorem 2.1 in \cite{Fuk}, $X=Y/G.$ Since the displacements of the
(nontrivial) deck transformations by $\pi _{1}\left( X_{i}\right) $ on $%
\tilde{X}_{i}\longrightarrow X_{i}$ are uniformly bounded from below, the
action by $G$ on $Y$ is properly discontinuous. Hence $Y\longrightarrow
Y/G=X $ is a covering space of $X$. By the Stability Theorem, $Y$ is simply
connected, so $Y$ is the universal cover of $X.$
\end{proof}

\begin{remark}
When the $X_{i}$ are Riemannian manifolds, one can get the uniform lower
bound for the systoles of the $X_{i}$s from the generalized Butterfly Lemma
in \cite{GrovPet1}. The same argument also works in the Alexandrov case but
requires Perelman's critical point theory, and hence is no simpler than what
we presented above.
\end{remark}

Lens spaces show that without the noncollapsing hypothesis this result is
false even in constant curvature.

%
%

\section{Cross Cap Stability\label{Cross Cap Stab}}

The main step to prove Theorem \ref{Cross Cap Stability} is the following.

\begin{theorem}
\label{cross cap embedding}Let $\left\{ M^{\alpha }\right\} _{\alpha
=1}^{\infty }$ be a sequence of closed Riemannian $n$--manifolds with $%
\mathrm{sec}$ $M^{\alpha }\geq k$ so that 
\begin{equation*}
M^{\alpha }\longrightarrow C_{k,r}^{n}
\end{equation*}%
in the Gromov-Hausdorff topology. Let $\tilde{M}^{\alpha }$ be the universal
cover of $M^{\alpha }.$ Then for all but finitely many $\alpha ,$ there is a 
$C^{1}$ embedding 
\begin{equation*}
\tilde{M}^{\alpha }\hookrightarrow \mathbb{R}^{n+1}\setminus \left\{
0\right\}
\end{equation*}%
that is equivariant with respect to the deck transformations of $\tilde{M}%
^{\alpha }\longrightarrow M^{\alpha }$ and the $Z_{2}$--action on $R^{n+1}$
generated by $-id.${\Large \ }
\end{theorem}

Two and three manifolds have unique differential structures up to
diffeomorphism; so in dimensions two and three Theorems \ref{Cross Cap
Stability} and \ref{cross cap embedding} follow from the main result of \cite%
{GrovPet3}. We give the proof in dimension 4 in section 6. Until then, we
assume that $n\geq 5.$

\begin{proof}[Proof of Theorem \protect\ref{Cross Cap Stability} modulo
Theorem \protect\ref{cross cap embedding}.]
By Perelman's Stability Theorem all but finitely many $\{\tilde{M}^{\alpha
}\}_{\alpha =1}^{\infty }$ are homeomorphic to $S^{n}$ (cf \cite{GrovPet3}).
Combining this with Theorem \ref{cross cap embedding} and Brown's Theorem
9.7 in \cite{Mil1} gives an H--cobordism between the embedded image of $%
\tilde{M}^{\alpha }\subset \mathbb{R}^{n+1}$ and the standard $S^{n}.$
Modding out by $\mathbb{Z}_{2},$ we see that $M^{\alpha }$ and $\mathbb{R}%
P^{n}$ are H--cobordant. Since the Whitehead group of $\mathbb{Z}_{2}$ is
trivial ( \cite{Hig}, \cite{Mil2}, p. 373), any H--cobordism between $%
M_{\alpha }$ and $\mathbb{R}P^{n}$ is an S--cobordism and hence a product,
which completes the proof. \cite{Bar, Maz, Stal}
\end{proof}

The proof of Theorem \ref{Cross Cap Stability} does not exploit any a priori
differential structure on the Crosscap. Instead we exploit a model embedding
of the double disk 
\begin{equation*}
\mathbb{D}_{k}^{n}\left( r\right) \hookrightarrow \mathbb{R}^{n+1},
\end{equation*}%
whose restriction to either half, $\mathcal{D}_{k}^{n}\left( r\right) ^{+}$
or $\mathcal{D}_{k}^{n}\left( r\right) ^{-}$, is the identity on the last $n$%
--coordinates. By describing the identity $\mathcal{D}_{k}^{n}\left(
r\right) \longrightarrow \mathcal{D}_{k}^{n}\left( r\right) $ in terms of
distance functions, we then argue that this embedding can be lifted to all
but finitely many of a sequence $\{M^{\alpha }\}$ converging to $\mathbb{D}%
_{k}^{n}\left( r\right) .$

\subsection*{The Model Embedding}

Let $A:\mathbb{D}_{k}^{n}\left( r\right) \rightarrow \mathbb{D}%
_{k}^{n}\left( r\right) $ be the free involution mentioned in Example \ref%
{crosscapexample}. For $z\in \mathbb{D}_{k}^{n}\left( r\right) ,$ we define $%
f_{z}:\mathbb{D}_{k}^{n}(r)\rightarrow \mathbb{R}$ by 
\begin{equation}
f_{z}(x)=h_{k}\circ \mathrm{dist}\left( A\left( z\right) ,x\right)
-h_{k}\circ \mathrm{dist}\left( z,x\right)  \label{definitionoffi}
\end{equation}%
where $h_{k}:\mathbb{R}\rightarrow \mathbb{R}$ is defined as 
\begin{equation*}
h_{k}(x)=\left\{ 
\begin{array}{cc}
\frac{1}{2\sinh r}\cosh (x) & \text{ if }k=-1 \\ 
\frac{x^{2}}{4r} & \text{ if }k=0 \\ 
\frac{1}{2\sin r}\cos (x) & \text{ if }k=1.%
\end{array}%
\right.
\end{equation*}

Recall that we view $\mathcal{D}_{k}^{n}\left( r\right) ^{\pm }$ as metric $%
r $-balls centered at $p_{0}=e_{0}$ and $A(p_{0})=-e_{0}$ in either $H_{\pm
}^{n}$, $\{\pm e_{0}\}\times \mathbb{R}^{n},$ or $S^{n}$. For $i=1,2,\ldots
,n$ we set 
\begin{equation}
p_{i}:=\left\{ 
\begin{array}{cc}
\cosh (r)e_{0}+\sinh (r)e_{i} & \text{ if }k=-1 \\ 
e_{0}+re_{i} & \text{ if }k=0 \\ 
\cos (r)e_{0}-\sin (r)e_{i} & \text{ if }k=1.%
\end{array}%
\right.  \label{definitionofpi}
\end{equation}%
The functions $\{f_{i}\}_{i=1}^{n}:=\{f_{p_{i}}\}_{i=1}^{n}$ are then
restrictions of the last $n$--coordinate functions of $\mathbb{R}^{n+1}$ to $%
\mathcal{D}_{k}^{n}\left( r\right) ^{\pm }.$ We set $f_{0}:=f_{p_{0}}$. In
contrast to $f_{1},\ldots ,f_{n}$, our $f_{0}$ is not a coordinate function.
On the other hand its gradient is well defined everywhere on $\mathbb{D}%
_{k}^{n}\left( r\right) \setminus \left\{ p_{0},A\left( p_{0}\right)
\right\} ,$ even on $\partial \mathcal{D}_{k}^{n}\left( r\right)
^{+}=\partial \mathcal{D}_{k}^{n}\left( r\right) ^{-}$ where it is normal to 
$\partial \mathcal{D}_{k}^{n}\left( r\right) ^{+}=\partial \mathcal{D}%
_{k}^{n}\left( r\right) ^{-}.$

Define $\Phi :\mathbb{D}_{k}^{n}\left( r\right) \rightarrow \mathbb{R}^{n+1}$%
, by 
\begin{equation*}
\Phi =\left( f_0,f_{1},f_{2},\cdots ,f_{n}\right) ,
\end{equation*}%
and observe that

\begin{proposition}
$\Phi $ is a continuous, $\mathbb{Z}_{2}$--equivariant embedding.
\end{proposition}

\begin{proof}
Write $\mathbb{R}^{n+1}=\mathbb{R\times R}^{n}$ and let $\pi :\mathbb{%
R\times R}^{n}\rightarrow \mathbb{R}^{n}$ be projection. Since $%
f_{1},f_{2},\cdots ,f_{n}$ are coordinate functions, the restrictions 
\begin{equation*}
\pi \circ \Phi |_{\mathcal{D}_{k}^{n}\left( r\right) ^{\pm }}:\mathcal{D}%
_{k}^{n}\left( r\right) ^{\pm }\longrightarrow \mathbb{R}^{n}
\end{equation*}%
are both the identity. From this and the definition of $f_{0},$ we conclude
that $\Phi $ is one--to--one. Since $\mathbb{D}_{k}^{n}\left( r\right) $ is
compact, it follows that $\Phi $ is an embedding. The $\mathbb{Z}_{2}$%
--equivariance is immediate from definition \ref{definitionoffi}.
\end{proof}

\subsection*{Lifting the Model Embedding}

To start the proof of Theorem \ref{cross cap embedding} let $\left\{
M^{\alpha }\right\} _{\alpha =1}^{\infty }$ be a sequence of closed
Riemannian $n$--manifolds with $\mathrm{sec}$ $M^{\alpha }\geq k$ so that 
\begin{equation*}
M^{\alpha }\longrightarrow C_{k,r}^{n},
\end{equation*}%
and we let $\{\tilde{M}^{\alpha }\}_{\alpha =1}^{\infty }$ denote the
corresponding sequence of universal covers. From Theorem \ref{covers
converge}, a subsequence of $\{\tilde{M}^{\alpha }\}_{\alpha =1}^{\infty }$
together with the deck transformations $\tilde{M}^{\alpha }\longrightarrow
M^{\alpha }$ converge to $\left( \mathbb{D}_{k}^{n}(r),A\right) .$ For all
but finitely many $\alpha ,$ $\pi _{1}\left( M^{\alpha }\right) $ is
isomorphic to $\mathbb{Z}_{2}.$ We abuse notation and call the nontrivial
deck transformation of $\tilde{M}^{\alpha }\longrightarrow M^{\alpha }$, $A.$

First we extend definition \ref{definitionoffi} by letting $f_{z}^{\alpha }:%
\tilde{M}^{\alpha }\rightarrow \mathbb{R}$ be defined by 
\begin{equation}
f_{z}^{\alpha }(x)=h_{k}\circ \mathrm{dist}(A(z),x)-h_{k}\circ \mathrm{dist}%
(z,x).  \label{defoffalpha}
\end{equation}%
Let $p_{i}^{\alpha }\in \tilde{M}^{\alpha }$ converge to $p_{i}\in \mathbb{D}%
_{k}^{n}(r),$ and for some $d>0$ define $f_{i,d}^{\alpha }:\tilde{M}^{\alpha
}\rightarrow \mathbb{R}$ by 
\begin{equation}
f_{i,d}^{\alpha }(x)=\frac{1}{\mathrm{vol\,}B(p_{i}^{\alpha },d)}%
\int_{q^{\alpha }\in B(p_{i}^{\alpha },d)}f_{q^{\alpha }}^{\alpha }(x).
\label{bar--f--alpha}
\end{equation}%
Differentiation under the integral gives

\begin{proposition}
\label{f_i C^1}The $f_{i,d}^{\alpha }$ are $C^{1}$ and $\left\vert \nabla
f_{i,d}^{\alpha }\right\vert \leq 2.$
\end{proposition}

We now define $\Phi _{d}^{\alpha }:\tilde{M}^{\alpha }\rightarrow \mathbb{R}%
^{n+1}$ by 
\begin{equation*}
\Phi _{d}^{\alpha }=\left( f_{0,d}^{\alpha },f_{1,d}^{\alpha
},f_{2,d}^{\alpha },\cdots ,f_{n,d}^{\alpha }\right) .
\end{equation*}%
As $\alpha \rightarrow \infty $ and $d\rightarrow 0$, $\Phi _{d}^{\alpha }$
converges to $\Phi $ in the Gromov--Hausdorff sense. Since $\Phi $ is an
embedding it follows that $\Phi _{d}^{\alpha }$ is one--to--one in the
large. More precisely,

\begin{proposition}
For any $\nu >0,$ if $\alpha $ is sufficiently large and $d$ is sufficiently
small, then 
\begin{equation*}
\Phi _{d}^{\alpha }\left( x\right) \neq \Phi _{d}^{\alpha }\left( y\right) ,
\end{equation*}%
provided $\mathrm{dist}\left( x,y\right) >\nu .$
\end{proposition}

Since the $\mathbb{Z}_{2}$-equivariance of $\Phi^{\alpha }_d$ immediately
follows from definition \ref{bar--f--alpha}, all that remains to prove
Theorem \ref{cross cap embedding} is the following proposition:

\begin{proposition}
\label{uniform immersion}There is a $\rho >0$ so that $\Phi _{d}^{\alpha }$
is one to one on all $\rho $--balls, provided that $\alpha $ is sufficiently
large and $d$ is sufficiently small.
\end{proposition}

This is a consequence of Key Lemma \ref{CC Key Lemma} (stated below), whose
statement and proof occupy the remainder of this section.

\subsection*{Uniform Immersion}

The proof of the Inverse Function Theorem in \cite{Rud} gives

\begin{theorem}
(Quantitative Immersion Theorem) Let 
\begin{equation*}
\mathbb{R}_{\hat{\imath}}^{n}:=\left\{ \left( x_{1},x_{2},\ldots
,x_{i-1},0,x_{i+1},\ldots ,x_{n+1}\right) \right\} \subset \mathbb{R}^{n+1}
\end{equation*}%
and let 
\begin{equation*}
P_{\hat{\imath}}:\mathbb{R}^{n+1}\longrightarrow \mathbb{R}_{\hat{\imath}%
}^{n}
\end{equation*}%
be orthogonal projection.

Let $F:\mathbb{R}^{n}\longrightarrow \mathbb{R}^{n+1}$ be a $C^{1}$ map so
that for some $a\in \mathbb{R}^{n},$ $\lambda >0,$ and $\rho >0,$ there is
an $i\in \left\{ 1,\ldots ,n+1\right\} $ so that 
\begin{equation*}
\left\vert d\left( P_{\hat{\imath}}\circ F\right) _{a}\left( v\right)
\right\vert \geq \lambda \left\vert v\right\vert
\end{equation*}%
and 
\begin{equation*}
\left\vert d\left( P_{\hat{\imath}}\circ F\right) _{a}\left( v\right)
-d\left( P_{\hat{\imath}}\circ F\right) _{x}\left( v\right) \right\vert <%
\frac{\lambda }{2}\left\vert v\right\vert
\end{equation*}%
for all $x\in B\left( a,\rho \right) $ and $v\in \mathbb{R}^{n},$ then $%
\left( P_{\hat{\imath}}\circ F\right) |_{B\left( a,\rho \right) }$ is a
one--to--one, open map.
\end{theorem}

We note that every space of directions to $\mathbb{D}_{k}^{n}(r)$ is
isometric to $S^{n-1}.$ By proposition \ref{posdeltastrainedrad}, there are $%
r,\delta >0$ so that every point in the double disk has a neighborhood $B$
that is $\left( n,\delta ,r\right) $--strained. If $B\subset \mathbb{D}%
_{k}^{n}(r)$ is $(n,\delta ,r)$--strained by $\{a_{i},b_{i}\}_{i=1}^{n}$, by
continuity of comparison angles, we may assume there are sets $B^{\alpha
}\subset \tilde{M}^{\alpha }$ $(n,\delta ,r)$--strained by $\{a_{i}^{\alpha
},b_{i}^{\alpha }\}_{i=1}^{n}$ such that 
\begin{equation*}
\left( \left\{ \left( a_{i}^{\alpha },b_{i}^{\alpha }\right) \right\}
_{i=1}^{n},B^{\alpha }\right) \longrightarrow \left( \left\{ \left(
a_{i},b_{i}\right) \right\} _{i=1}^{n},B\right) .
\end{equation*}%
Given $x^{\alpha }\in B^{\alpha },$ we let $\varphi _{x^{\alpha }}^{\eta }$
be as in \ref{Otsu-Shioya}.

To prove Proposition \ref{uniform immersion} it suffices to prove the
following.

\begin{keylemma}
\label{CC Key Lemma}There is a $\lambda >0$ and $\rho >0$ so that for all $%
x^{\alpha }\in \tilde{M}^{\alpha }$ there is an $i_{x^{\alpha }}\in \left\{
0,1,\ldots ,n\right\} $ such that the function $F:=\Phi _{d}^{\alpha }\circ
\left( {\varphi _{{x^{\alpha }}}^{\eta }}\right) ^{-1}$ satisfies

\begin{enumerate}
\item 
\begin{equation*}
\left\vert d(P_{\hat{\imath}_{x^{\alpha }}}\circ F)_{\varphi _{{x^{\alpha }}%
}^{\eta }\left( x^{\alpha }\right) }\left( v\right) \right\vert >\lambda
\left\vert v\right\vert
\end{equation*}%
and

\item 
\begin{equation*}
\left\vert d\left( P_{\hat{\imath}_{x^{\alpha }}}\circ F\right) _{\varphi
_{x^{\alpha }}^{\eta }\left( y\right) }\left( v\right) -d\left( P_{\hat{%
\imath}_{x^{\alpha }}}\circ F\right) _{\varphi _{x^{\alpha }}^{\eta }\left(
x^{\alpha }\right) }\left( v\right) \right\vert <\frac{\lambda }{2}%
\left\vert v\right\vert
\end{equation*}%
for all $y\in B\left( x^{\alpha },\rho \right) $ and $v\in \mathbb{R}^{n}$,
provided that $\alpha $ is sufficiently large and $d$ and $\eta $ are
sufficiently small.
\end{enumerate}
\end{keylemma}

We show in the next subsection that part 1 of Key Lemma \ref{CC Key Lemma}
holds, and in the following subsection we show that part 2 holds.

\subsection*{Lower bound on the differential\label{lower bound on
differential}}

We begin by illustrating that, in a sense, the first part of the key lemma
holds for the model embedding.

\begin{lemma}
There is a $\lambda >0$ so that for all $v\in T\mathbb{D}_{k}^{n}(r)$ there
is a $j\left( v\right) \in \left\{ 0,1,\ldots ,n\right\} $ so that 
\begin{equation*}
\left\vert D_{v}f_{j\left( v\right) }\right\vert >\lambda \left\vert
v\right\vert .
\end{equation*}
\end{lemma}

\begin{proof}
Recall that the double disk $\mathbb{D}_{k}^{n}(r)$ is the union of two
copies of $\mathcal{D}_{k}^{n}(r)$ that we call $\mathcal{D}_{k}^{n}(r)^{+}$
and $\mathcal{D}_{k}^{n}(r)^{-}$---glued along their common boundary---that
throughout this section we call $\mathcal{S}:=\partial \mathcal{D}%
_{k}^{n}(r)^{\pm }.$

If $x\in \mathbb{D}_{k}^{n}(r)\setminus \mathcal{S}$, then for $i\neq 0,$ $%
\nabla f_{i}$ is unambiguously defined; moreover, 
\begin{equation*}
\left\{ \nabla f_{i}\left( x\right) \right\} _{i=1}^{n}
\end{equation*}%
is an orthonormal basis. Thus the lemma certainly holds on $\mathbb{D}%
_{k}^{n}(r)\setminus \mathcal{S}.$

For $x\in \mathcal{S}$ and $i\in \left\{ 1,\ldots ,n\right\} ,$ we can think
of the gradient of $f_{i}$ as multivalued. More precisely, for $x\in 
\mathcal{S},$ we view 
\begin{equation*}
\mathcal{S}\subset \mathcal{D}_{k}^{n}(r)^{\pm }\subset \left\{ 
\begin{array}{cc}
H_{\pm }^{n} & \text{if }k=-1 \\ 
\{\pm e_{0}\}\times \mathbb{R}^{n} & \text{if }k=0 \\ 
S^{n} & \text{if }k=1%
\end{array}%
\right.
\end{equation*}%
and define $\nabla f_{i}^{\pm }$ to be the gradient at $x$ of the coordinate
function that extends $f_{i}$ to either $H_{\pm }^{n},\{\pm e_{0}\}\times 
\mathbb{R}^{n},$ or $S^{n}.$

From definition \ref{definitionoffi}, for any $v\in T_{x}\mathbb{D}%
_{k}^{n}(r)$ 
\begin{equation*}
D_{v}f_{i}=\left\{ 
\begin{array}{cc}
\left\langle \nabla f_{i}^{+},v\right\rangle & \text{if }v\text{ is inward
to }\mathcal{D}_{k}^{n}(r)^{+} \\ 
\left\langle \nabla f_{i}^{-},v\right\rangle & \text{if }v\text{ is inward
to }\mathcal{D}_{k}^{n}(r)^{-}.%
\end{array}%
\right.
\end{equation*}%
Notice that the projections of $\nabla f_{i}^{+}$ and $\nabla f_{i}^{-}$
onto $T_{x}\mathcal{S}$ coincide, so for $v\in T_{x}\mathcal{S}$ we have $%
D_{v}f_{i}=\left\langle \nabla f_{i}^{+},v\right\rangle =\left\langle \nabla
f_{i}^{-},v\right\rangle .$ As $\left\{ \nabla f_{i}^{+}\right\} _{i=1}^{n}$
is an orthonormal basis, the lemma holds for $v\in T\mathcal{S}$ and hence
also for $v$ in a neighborhood $U$ of $T\mathcal{S}\subset T\mathbb{D}%
_{k}^{n}(r)|_{\mathcal{S}}.$ Since $\nabla f_{0}$ is well defined on $%
\mathcal{S}$ and normal to $\mathcal{S}$, for any unit $v\in T\mathbb{D}%
_{k}^{n}(r)|_{\mathcal{S}}\setminus U,$ we have $\left\vert
D_{v}f_{0}\right\vert >0.$ The lemma follows from the compactness of the set
of unit vectors in $T\mathbb{D}_{k}^{n}(r)|_{\mathcal{S}}\setminus U.$
\end{proof}

Notice that at $p_{k}$ and $A\left( p_{k}\right) $ the gradients of $f_{k}$
and $f_{0}$ are colinear. Using this we conclude

\begin{addendum}
\label{1st addendum}Let $p_{k}$ be any of $p_{1},\ldots p_{n}.$ There is an $%
\varepsilon >0$ so that for all $x\in B\left( p_{k},\varepsilon \right) \cup
B\left( A\left( p_{k}\right) ,\varepsilon \right) $ and all $v\in T_{x}%
\mathbb{D}_{k}^{n}(r)$, the index $j\left( v\right) $ in the previous lemma
can be chosen to be different from $k.$
\end{addendum}

\begin{lemma}
There is a $\lambda >0$ so that for all $v\in T_{x}\mathbb{D}_{k}^{n}(r)$
there is a $j\left( v\right) \in \left\{ 0,1,\ldots ,n\right\} $ so that%
\begin{equation*}
\left\vert D_{v}f_{z}\right\vert >\lambda \left\vert v\right\vert
\end{equation*}%
for all $z\in B(p_{j(v)},d),$ provided $d$ is sufficiently small.
\end{lemma}

\begin{proof}
If not then for each $i=0,1,\ldots ,n$ there is a sequence $%
\{z_{i}^{j}\}_{j=1}^{\infty }\subset \mathbb{D}^n_{k}(r)$ with $\mathrm{dist}%
(z_{i}^{j},p_{i})<\frac{1}{j}$ and a sequence of unit $v^{j}\in T_{x^{j}}%
\mathbb{D}_{k}^{n}(r)$ so that 
\begin{equation*}
\left\vert D_{v^{j}}f_{z_{i}^{j}}\right\vert <\frac{1}{j}.
\end{equation*}%
Choose the segments $x^{j}z_{i}^{j}$ and $x^{j}A\left( z_{i}^{j}\right) $ so
that 
\begin{eqnarray*}
\sphericalangle \left( \uparrow _{x^{j}}^{z_{i}^{j}},\text{ }v^{j}\right)
&=&\sphericalangle \left( \Uparrow _{x^{j}}^{z_{i}^{j}},\text{ }v^{j}\right) 
\text{ and } \\
\sphericalangle \left( \uparrow _{x^{j}}^{A\left( z_{i}^{j}\right) },\text{ }%
v^{j}\right) &=&\sphericalangle \left( \Uparrow _{x^{j}}^{A\left(
z_{i}^{j}\right) },\text{ }v^{j}\right) .
\end{eqnarray*}
After passing to subsequences, we have $v^{j}\rightarrow v$, $%
x^{j}\rightarrow x$ and 
\begin{eqnarray*}
x^{j}z_{i}^{j} &\rightarrow &xp_{i} \\
x^{j}A\left( z_{i}^{j}\right) &\rightarrow &xA\left( p_{i}\right) ,
\end{eqnarray*}%
for some choice of segments $xp_{i}$ and $xA\left( p_{i}\right) .$ Using
Lemma \ref{angle convergence} and Corollary \ref{convergence of vectors} we
conclude 
\begin{eqnarray}
\left\vert \sphericalangle \left( \uparrow _{x^{j}}^{z_{i}^{j}},v^{j}\right)
-\sphericalangle \left( \uparrow _{x}^{p_{i}},v\right) \right\vert &<&\tau
\left( \delta ,\tau \left( \left. \frac{1}{j}\right\vert \mathrm{dist}\left(
x,p_{i}\right) \right) \right) ,  \notag \\
\left\vert \sphericalangle \left( \uparrow _{x^{j}}^{A\left(
z_{i}^{j}\right) },v^{j}\right) -\sphericalangle \left( \uparrow
_{x}^{A\left( p_{i}\right) },v\right) \right\vert &<&\tau \left( \delta
,\tau \left( \left. \frac{1}{j}\right\vert \mathrm{dist}\left( x,A\left(
p_{i}\right) \right) \right) \right) .  \notag \\
&&  \label{close to v (2)}
\end{eqnarray}

If $x\notin \mathcal{S},$ then the segments $xp_{i}$ and $xA\left(
p_{i}\right) $ are unambiguously defined, and so the previous inequality and
the hypothesis $\left\vert D_{v^{j}}f_{z_{i}^{j}}\right\vert <\frac{1}{j},$
contradict the previous lemma and its addendum.

If $x\in \mathcal{S}$ and $v\in T_{x}\mathcal{S},$ then 
\begin{equation*}
\sphericalangle \left( \uparrow _{x}^{p_{i}},v\right) \text{ and }%
\sphericalangle \left( \uparrow _{x}^{A\left( p_{i}\right) },v\right)
\end{equation*}%
are independent of the choice of the segments $xp_{i}$ and $xA\left(
p_{i}\right) ,$ so the hypothesis $\left\vert
D_{v^{j}}f_{z_{i}^{j}}\right\vert <\frac{1}{j}$ together with the
Inequalities \ref{close to v (2)} contradict the previous lemma and its
addendum. Thus our result holds for $v\in T\mathcal{S}$ and hence also for $%
v $ in a neighborhood $U$ of $T\mathcal{S}\subset \left. T\mathbb{D}%
_{k}^{n}(r)\right\vert _{\mathcal{S}}.$

For a unit vector $v\in \left. T\mathbb{D}_{k}^{n}(r)\right\vert _{\mathcal{S%
}}\setminus U,$ we saw in the proof of the previous lemma that for some $%
\lambda >0$%
\begin{equation}
\left\vert D_{v}f_{0}\right\vert >\lambda .  \label{Big f_0 derivative}
\end{equation}%
For $x\in \mathcal{S},$ we have unique segments $xp_{0}$ and $xA\left(
p_{0}\right) ,$ so the hypothesis $\left\vert
D_{v^{j}}f_{z_{i}^{j}}\right\vert <\frac{1}{j}$ and inequalities \ref{close
to v (2)} contradict Inequality \ref{Big f_0 derivative}.
\end{proof}

Combining the proof of the previous lemma with Addendum \ref{1st addendum},
we get

\begin{addendum}
\label{2nd addendum}Let $p_{k}$ be any of $p_{1},\ldots p_{n}.$ There is an $%
\varepsilon >0$ so that for all $x\in B\left( p_{k},\varepsilon \right) \cup
B\left( A\left( p_{k}\right) ,\varepsilon \right) $ and all $v\in T_{x}%
\mathbb{D}_{k}^{n}(r)$, the index $j\left( v\right) $ in the previous lemma
can be chosen to be different from $k.$
\end{addendum}

\begin{lemma}
\label{Directional derivative lower bound}There is a $\lambda >0$ so that
for all $v\in T\tilde{M}^{\alpha }$ there is a $j\left( v\right) \in \left\{
0,1,\ldots ,n\right\} $ so that%
\begin{equation*}
D_{v}f_{j\left( v\right) ,d}^{\alpha }>\lambda \left\vert v\right\vert ,
\end{equation*}%
provided $\alpha $ is sufficiently large and $d$ is sufficiently small.
\end{lemma}

\begin{proof}
If the lemma were false, then there would be a sequence of unit vectors $%
\left\{ v^{\alpha }\right\} _{\alpha =1}^{\infty }$ with $v^{\alpha }\in
T_{x^{\alpha }}\tilde{M}^{\alpha }$ such that for all $i,$%
\begin{equation*}
\left\vert D_{v^{\alpha }}f_{i,d}^{\alpha }\right\vert <\tau \left( \frac{1}{%
\alpha },d\right) .
\end{equation*}%
Let $\lim_{\alpha \rightarrow \infty }x^{\alpha }=x\in \mathbb{D}_{k}^{n}(r)$%
. By Corollary \ref{convergence of vectors}, for any $\mu >0$ there is a
sequence $\left\{ w^{\alpha }\right\} _{\alpha =1}^{\infty }$ with $%
w^{\alpha }\in \Sigma _{x^{\alpha }}^{\mu }$ such that 
\begin{equation*}
\sphericalangle \left( v^{\alpha },w^{\alpha }\right) <\tau \left( \delta
,\mu \right) .
\end{equation*}%
Since $\left\vert \nabla f_{i,d}^{\alpha }\right\vert \leq 2,$ 
\begin{equation}
\left\vert D_{w^{\alpha }}f_{i,d}^{\alpha }\right\vert <\tau \left( \delta
,\mu ,\frac{1}{\alpha },d\right)  \label{smaller deriviative}
\end{equation}%
for all $i$. After passing to a subsequence, we conclude that $\left\{
\gamma _{w^{\alpha }}|_{\left[ 0,\mu \right] }\right\} _{\alpha =1}^{\infty
} $ converges to a segment $\gamma _{w}|_{\left[ 0,\mu \right] }.$ By the
previous lemma, there is a $\lambda >0$ and a $j\left( w\right) $ so that
for all $z\in B(p_{j(w)},d),$ 
\begin{equation}
\left\vert D_{w}f_{z}\right\vert >\lambda \left\vert w\right\vert ,
\label{big derivative}
\end{equation}%
provided $d$ is small enough. Moreover, by Addendum \ref{2nd addendum} we
may assume that 
\begin{eqnarray}
\mathrm{dist}\left( x,p_{j(w)}\right) &>&100d>\mu \text{ and}  \notag \\
\mathrm{dist}\left( x,A\left( p_{j(w)}\right) \right) &>&100d>\mu .\text{ %
\label{not dumb index}}
\end{eqnarray}%
By the Mean Value Theorem, there is a $z_{j\left( w\right) }^{\alpha }\in
B\left( p_{j\left( w\right) }^{\alpha },d\right) $ with 
\begin{equation}
D_{w^{\alpha }}f_{z_{j\left( w\right) }^{\alpha }}^{\alpha }=D_{w^{\alpha
}}f_{j\left( w\right) ,d}^{\alpha }.  \label{Same Derivative}
\end{equation}%
Choose segments $x^{\alpha }z_{j\left( w\right) }^{\alpha }$ and $x^{\alpha }%
{A}(z_{j\left( w\right) }^{\alpha })$ in $\tilde{M}^{\alpha }$ so that 
\begin{eqnarray*}
\sphericalangle \left( \uparrow _{x^{\alpha }}^{z_{j\left( w\right)
}^{\alpha }},w^{\alpha }\right) &=&\sphericalangle \left( \Uparrow
_{x^{\alpha }}^{z_{j\left( w\right) }^{\alpha }},w^{\alpha }\right) \text{
and} \\
\sphericalangle \left( \uparrow _{x^{\alpha }}^{{A}(z_{j\left( w\right)
}^{\alpha })},w^{\alpha }\right) &=&\sphericalangle \left( \Uparrow
_{x^{\alpha }}^{{A}(z_{j\left( w\right) }^{\alpha })},w^{\alpha }\right) .
\end{eqnarray*}%
After passing to a subsequence, we may assume that for some $z_{j\left(
w\right) }\in B(p_{j\left( w\right) },d)$, $x^{\alpha }z_{j\left( w\right)
}^{\alpha }$ and $x^{\alpha }{A}(z_{j\left( w\right) }^{\alpha })$ converge
to segments $xz_{j\left( w\right) }$ and $x{A}(z_{j\left( w\right) })$,
respectively. By Lemma \ref{angle convergence}, 
\begin{eqnarray*}
\left\vert \sphericalangle ( \uparrow _{x^{\alpha }}^{z_{j\left( w\right)
}^{\alpha }},\gamma _{w^{\alpha }}^{\prime }\left( 0\right))
-\sphericalangle ( \uparrow _{x}^{z_{j\left( w\right) }},\gamma _{w}^{\prime
}\left( 0\right)) \right\vert &<&\tau \left( \delta ,\tau \left( {1/\alpha }%
|\mu ,\mathrm{dist}\left( x,z_{j\left( w\right) }\right) \right) \right) \\
\left\vert \sphericalangle ( \uparrow _{x^{\alpha }}^{{A}(z_{j\left(
w\right) }^{\alpha })},\gamma _{w^{\alpha }}^{\prime }\left( 0\right) )
-\sphericalangle ( \uparrow _{x}^{A\left( z_{j\left( w\right) }\right)
},\gamma _{w}^{\prime }\left( 0\right)) \right\vert &<&\tau \left( \delta
,\tau \left( {1/\alpha }|\mu ,\mathrm{dist}\left( x,A\left( z_{j\left(
w\right) }\right) \right) \right) \right).
\end{eqnarray*}%
Combining the previous two sets of displays with \ref{not dumb index} 
\begin{equation}
\left\vert D_{w^{\alpha }}f_{z_{j\left( w\right) }^{\alpha }}^{\alpha
}-D_{w}f_{z_{j\left( w\right) }}\right\vert <\tau \left( \delta ,\tau \left( 
{1/\alpha }|\mu \right) \right) .  \label{derivative convergence}
\end{equation}%
So by Equation \ref{Same Derivative}, 
\begin{equation*}
\left\vert D_{w^{\alpha }}f_{j\left( w\right) ,d}^{\alpha
}-D_{w}f_{z_{j\left( w\right) }}\right\vert <\tau \left( \delta ,\tau \left( 
{1/\alpha }|\mu \right) \right) ,
\end{equation*}%
but this contradicts Inequalities \ref{smaller deriviative} and \ref{big
derivative}.
\end{proof}

The first claim of Key Lemma \ref{CC Key Lemma} follows by combining the
previous lemma with the fact that the differentials of the $\varphi
_{x^{\alpha }}^{\eta }$'s are almost isometries.

\begin{remark}
\label{index choice}Note that when $x^{\alpha }$ is close to $p_{k}$ or $%
A\left( p_{k}\right) ,$ the desired estimate%
\begin{equation*}
\left\vert d(P_{\hat{\imath}_{x^{\alpha }}}\circ F)_{\varphi _{{x^{\alpha }}%
}^{\eta }\left( x^{\alpha }\right) }\left( v\right) \right\vert >\lambda
\left\vert v\right\vert
\end{equation*}%
holds with $P_{\hat{\imath}_{x^{\alpha }}}=P_{\hat{k}}.$ This follows from
Addendum \ref{2nd addendum} and the proof of the previous lemma.
\end{remark}

\subsection*{ Equicontinuity of Differentials}

In this subsection, we establish the second part of the key lemma. If $%
x^{\alpha }$ is not close to one of the $p_{k}s$ or $A\left( p_{k}\right) s$
we will show the stronger estimate%
\begin{equation}
\left\vert d\left( F\right) _{\varphi _{x^{\alpha }}^{\eta }\left( y\right)
}\left( v\right) -d\left( F\right) _{\varphi _{x^{\alpha }}^{\eta }\left(
x^{\alpha }\right) }\left( v\right) \right\vert <\frac{\lambda }{2}%
\left\vert v\right\vert .  \label{away from base}
\end{equation}%
So at such points, the second part of the key lemma holds with \emph{any }%
choice of coordinate projection $P_{\hat{\imath}_{x^{\alpha }}}.$

For $x^{\alpha }$ close to $p_{k}$ or $A\left( p_{k}\right) ,$ we will show 
\begin{equation}
\left\vert d\left( P_{\hat{k}}\circ F\right) _{\varphi _{x^{\alpha }}^{\eta
}\left( y\right) }\left( v\right) -d\left( P_{\hat{k}}\circ F\right)
_{\varphi _{x^{\alpha }}^{\eta }\left( x^{\alpha }\right) }\left( v\right)
\right\vert <\frac{\lambda }{2}\left\vert v\right\vert ,  \label{at base}
\end{equation}%
where $\lambda $ is the constant whose existence was established in the
previous section. Together with remark \ref{index choice}, this will
establish the key lemma.

Suppose $B\subset \mathbb{D}_{k}^{n}(r)$ is $\left( n,\delta ,r\right) $%
--strained by $\left\{ \left( a_{i},b_{i}\right) \right\} _{i=1}^{n}$. Let $%
x,y\in B$ and let 
\begin{equation*}
\varphi ^{\eta }:B\longrightarrow \mathbb{R}^{n}
\end{equation*}%
be the map defined in \ref{Otsu-Shioya} and \cite{OS}. Set 
\begin{equation*}
P_{x,y}:=\left( d\varphi ^{\eta }\right) _{y}^{-1}\circ \left( d\varphi
^{\eta }\right) _{x}:T_{x}\mathbb{D}_{k}^{n}(r)\rightarrow T_{y}\mathbb{D}%
_{k}^{n}(r).
\end{equation*}%
It follows that $P_{x,y}$ is a $\tau (\delta ,\eta )$--isometry.

\begin{lemma}
\label{measure convergence}Let $B\subset \mathbb{D}_{k}^{n}(r)$ be $\left(
n,\delta ,r\right) $--strained by $\left\{ \left( a_{i},b_{i}\right)
\right\} _{i=1}^{n}.$ Given $\varepsilon >0$ and $x\in B,$ there is a $\rho
\left( x,\varepsilon \right) >0$ so that the following holds.

For all $k\in \left\{ 0,1,\ldots ,n\right\} ,$ there is a subset $%
E_{k,x}\subset \left\{ B\left( p_{k},d\right) \cup B\left( A\left(
p_{k}\right) ,d\right) \right\} $ with measure $\mu \left( E_{k,x}\right)
<\varepsilon $ so that for all $z\in B\left( p_{k},d\right) \setminus
E_{k,x} $, all $y\in B\left( x,\rho \left( x,\varepsilon \right) \right) ,$
and all $v\in \Sigma _{x},$ 
\begin{eqnarray*}
\left\vert \sphericalangle \left( v,\uparrow _{x}^{z}\right)
-\sphericalangle \left( P_{x,y}\left( v\right) ,\uparrow _{y}^{z}\right)
\right\vert &<&\tau \left( \left. \varepsilon ,\delta ,\eta \right\vert 
\mathrm{dist}\left( x,z\right) \right) \text{ and} \\
\left\vert \sphericalangle \left( v,\uparrow _{x}^{A\left( z\right) }\right)
-\sphericalangle \left( P_{x,y}\left( v\right) ,\uparrow _{y}^{A\left(
z\right) }\right) \right\vert &<&\tau \left( \varepsilon ,\delta ,\eta |%
\mathrm{dist}\left( x,A(z)\right) \right) .
\end{eqnarray*}
\end{lemma}

\begin{proof}
Let $C_{x}=\left\{ z|z\in \mathrm{Cutlocus}\left( x\right) \text{ or }%
A\left( z\right) \in \mathrm{Cutlocus}\left( x\right) \right\} $ and set 
\begin{equation*}
E_{k,x}=B\left( C_{x},\nu \right) \cap \left\{ B\left( p_{k},d\right) \cup
B\left( A\left( p_{k}\right) ,d\right) \right\} .
\end{equation*}%
Choose $\nu >0$ so that $\mu \left( E_{k,x}\right) <\varepsilon .$

By Proposition \ref{geodesic continuity}, for each $z\in B\left(
p_{k},d\right) \setminus E_{k,x},$ there is a $\rho \left( x,z,\varepsilon
\right) $ so that for all $y\in B\left( x,\rho \left( x,z,\varepsilon
\right) \right) $ and any choice of segment $zy,$ 
\begin{equation*}
\mathrm{dist}\left( zx,zy\right) <\varepsilon ,
\end{equation*}%
where $zx$ is the unique segment from $z$ to $x.$

Making $\rho \left( x,z,\varepsilon \right) $ smaller and using Corollary %
\ref{Angle continuity}, it follows that for any $\tilde{a}_{i},\bar{a}%
_{i}\in B\left( a_{i},\eta \right) ,$%
\begin{eqnarray*}
\left\vert \sphericalangle \left( \Uparrow _{x}^{\tilde{a}_{i}},\uparrow
_{x}^{z}\right) -\sphericalangle \left( \Uparrow _{y}^{\bar{a}_{i}},\uparrow
_{y}^{z}\right) \right\vert &<&\tau \left( \delta ,\varepsilon ,\eta |%
\mathrm{dist}\left( x,z\right) ,\mathrm{dist}\left( y,z\right) \right) \\
&=&\tau \left( \delta ,\varepsilon ,\eta |\mathrm{dist}\left( x,z\right)
\right) .
\end{eqnarray*}%
It follows that%
\begin{equation*}
\left\vert \left( d\varphi ^{\eta }\right) _{x}\left( \uparrow
_{x}^{z}\right) -\left( d\varphi ^{\eta }\right) _{y}\left( \uparrow
_{y}^{z}\right) \right\vert <\tau \left( \delta ,\varepsilon ,\eta |\mathrm{%
dist}\left( x,z\right) \right) ,
\end{equation*}%
and hence 
\begin{equation*}
\sphericalangle \left( P_{x,y}\left( \uparrow _{x}^{z}\right) ,\uparrow
_{y}^{z}\right) =\sphericalangle \left( \left( d\varphi ^{\eta }\right)
_{y}^{-1}\circ \left( d\varphi ^{\eta }\right) _{x}\left( \uparrow
_{x}^{z}\right) ,\left( \uparrow _{y}^{z}\right) \right) <\tau \left( \delta
,\varepsilon ,\eta |\mathrm{dist}\left( x,z\right) \right) .
\end{equation*}
So for any $v\in \Sigma _{x},$ 
\begin{eqnarray*}
\left\vert \sphericalangle \left( v,\uparrow _{x}^{z}\right)
-\sphericalangle \left( P_{x,y}\left( v\right) ,\uparrow _{y}^{z}\right)
\right\vert &\leq &\left\vert \sphericalangle \left( v,\uparrow
_{x}^{z}\right) -\sphericalangle \left( P_{x,y}\left( v\right)
,P_{x,y}\left( \uparrow _{x}^{z}\right) \right) \right\vert + \\
&&\left\vert \sphericalangle \left( P_{x,y}\left( v\right) ,P_{x,y}\left(
\uparrow _{x}^{z}\right) \right) -\sphericalangle \left( P_{x,y}\left(
v\right) ,\uparrow _{y}^{z}\right) \right\vert \\
&<&\tau \left( \delta ,\eta \right) +\tau \left( \varepsilon ,\delta ,\eta |%
\mathrm{dist}\left( x,z\right) \right) \\
&=&\tau \left( \varepsilon ,\delta ,\eta |\mathrm{dist}\left( x,z\right)
\right) .
\end{eqnarray*}

Using Proposition \ref{geodesic continuity} and the precompactness of $%
B\left( p_{k},d\right) \setminus E_{k,x},$ we can then choose $\rho \left(
x,z,\varepsilon \right) $ to be independent of $z\in B\left( p_{k},d\right)
\setminus E_{k,x}.$ A similar argument gives the second inequality.
\end{proof}

\begin{corollary}
Given any $\varepsilon >0$, there is a $\rho (\varepsilon )>0$ so that for
any $x\in \mathbb{D}_{k}^{n}(r)$, $y\in B(x,\rho (\varepsilon ))$, and $z\in
B(p_{i},d)\setminus E_{i,x},$ we have 
\begin{equation*}
\left\vert D_{v}f_{z}-D_{P_{x,y}\left( v\right) }f_{z}\right\vert <\tau
\left( \varepsilon ,\delta ,\eta |\mathrm{dist}\left( z,x\right) ,\mathrm{%
dist}\left( A\left( z\right) ,x\right) \right)
\end{equation*}%
for all unit vectors $v\in \Sigma _{x}$.
\end{corollary}

\begin{proof}
Since $\mathbb{D}_{k}^{n}(r)$ is compact, the $\rho (\varepsilon ,x)$ from
the previous lemma can be chosen to be independent of $x$.

Given $x\in \mathbb{D}_{k}^{n}(r)$, $y\in B(x,\rho (\varepsilon )),$ and $%
v\in \Sigma _{x},$ choose segments $yz$ and $yA(z)$ so that 
\begin{eqnarray*}
\sphericalangle \left( \uparrow _{y}^{z},P_{x,y}\left( v\right) \right)
&=&\sphericalangle \left( \Uparrow _{y}^{z},P_{x,y}\left( v\right) \right) 
\text{ and} \\
\sphericalangle \left( \uparrow _{y}^{A\left( z\right) },P_{x,y}\left(
v\right) \right) &=&\sphericalangle \left( \Uparrow _{y}^{A\left( z\right)
},P_{x,y}\left( v\right) \right) .
\end{eqnarray*}%
Since the segments $xz$ and $xA(z)$ are unique, the result follows from the
formula for directional derivatives of distance functions, the previous
lemma, and the chain rule.
\end{proof}

We can lift a strainer from $\mathbb{D}_{k}^{n}(r)$ to any $\tilde{M}%
^{\alpha }$ if \textrm{dist}$_{GH}\left( \tilde{M}^{\alpha },\mathbb{D}%
_{k}^{n}(r)\right) $ is sufficiently small. So if $x^{\alpha }\ $and $%
y^{\alpha }$ are sufficiently close, we define 
\begin{equation*}
P_{x^{\alpha },y^{\alpha }}:=\left( d\varphi ^{\eta }\right) _{y^{\alpha
}}^{-1}\circ \left( d\varphi ^{\eta }\right) _{x^{\alpha }}:T_{x^{\alpha }}%
\tilde{M}^{\alpha }\rightarrow T_{y^{\alpha }}\tilde{M}^{\alpha }.
\end{equation*}

\begin{lemma}
Let $i$ be in $\left\{ 0,\ldots ,n\right\} .$ There is a $\rho >0$ so that
for any $x^{\alpha }\in \tilde{M}^{\alpha }$, any $y^{\alpha }\in
B(x^{\alpha },\rho ),$ and any unit $v^{\alpha }\in T_{x^{\alpha }}\tilde{M}%
^{\alpha }$ we have 
\begin{equation*}
\left\vert D_{v^{\alpha }}f_{i,d}^{\alpha }-D_{P_{x^{\alpha },y^{a}}\left(
v^{\alpha }\right) }f_{i,d}^{\alpha }\right\vert <\tau \left( \rho ,\frac{1}{%
\alpha },\delta ,\eta |\mathrm{dist}\left( x^{\alpha },p_{i}^{\alpha
}\right) ,\mathrm{dist}\left( x^{\alpha },A\left( p_{i}^{\alpha }\right)
\right) \right) ,
\end{equation*}%
provided $d$ is sufficiently small.
\end{lemma}

\begin{proof}
If not, then for any $\rho >0$ and some $i=0,1,\ldots ,n$, there would be a
sequence of points $x^{\alpha }\rightarrow x\in \mathbb{D}_{k}^{n}(r)$, a
sequence of unit vectors $\left\{ v^{\alpha }\right\} _{\alpha =1}^{\infty }$
and a constant $C>0$ that is independent of $\alpha $, $\delta ,$ and $\eta $
so that 
\begin{eqnarray}
\left\vert D_{v^{\alpha }}f_{i,d}^{\alpha }-D_{P_{x^{\alpha },y^{a}}\left(
v^{\alpha }\right) }f_{i,d}^{\alpha }\right\vert &\geq &C,  \notag \\
\mathrm{dist}\left( x,p_{i}\right) &\geq &C,\text{ and}  \notag \\
\mathrm{dist}\left( x,A\left( p_{i}\right) \right) &\geq &C
\label{far from
p_i}
\end{eqnarray}%
for some $y^{\alpha }\in B(x^{\alpha },\rho )$. Choose $\varepsilon >0$ and
take $\rho <\rho (\varepsilon )$ where $\rho (\varepsilon )$ is from the
previous corollary. We assume $B\left( x,\rho (\varepsilon )\right) $ is $%
(n,\delta ,r)$--strained. Let $y=\lim y^{\alpha }$ and $\mu >0$ be
sufficiently small. By corollary \ref{convergence of vectors}, there are
sequences $\left\{ w^{\alpha }\right\} _{\alpha =1}^{\infty }\in \Sigma
_{x^{\alpha }}^{\mu }$ and $\left\{ \tilde{w}^{\alpha }\right\} _{\alpha
=1}^{\infty }\in \Sigma _{y^{\alpha }}^{\mu }$ so that 
\begin{eqnarray}
\sphericalangle \left( v^{\alpha },w^{\alpha }\right) &<&\tau \left( \delta
,\mu \right)  \notag \\
\sphericalangle \left( P_{x^{\alpha },y^{a}}\left( w^{\alpha }\right) ,%
\tilde{w}^{\alpha }\right) &<&\tau \left( \delta ,\mu \right)
\label{mu
move}
\end{eqnarray}%
and subsequences $\left\{ \gamma _{w^{\alpha }}\right\} _{\alpha =1}^{\infty
}$ and $\left\{ \gamma _{\tilde{w}^{\alpha }}\right\} _{\alpha =1}^{\infty }$
converging to segments $\gamma _{w}$ and $\gamma _{\tilde{w}}$ that are
parameterized on $\left[ 0,\mu \right] .$ Since $|\nabla f_{i,d}^{\alpha
}|\leq 2$, we may assume for a possibly smaller constant $C$ that 
\begin{equation*}
\left\vert D_{w^{\alpha }}f_{i,d}^{\alpha }-D_{\tilde{w}^{\alpha
}}f_{i,d}^{\alpha }\right\vert \geq C\text{.}
\end{equation*}
Thus for some $z^{\alpha }\in B(p_{i}^{\alpha },d)$ with $\mathrm{dist}_{%
\mathrm{Haus}}\left( z^{\alpha },E_{i,x}\right) >2\nu ,$ 
\begin{equation}
\left\vert D_{w^{\alpha }}f_{z^{\alpha }}^{\alpha }-D_{\tilde{w}^{\alpha
}}f_{z^{\alpha }}^{\alpha }\right\vert \geq \frac{C}{2}.  \label{big gap}
\end{equation}

Passing to a subsequence, we have $z^{\alpha }\rightarrow z\in
B(p_{i},d)\setminus E_{i,x}.$ As in the proof of Lemma \ref{Directional
derivative lower bound} (Inequality \ref{derivative convergence}), we have 
\begin{eqnarray*}
\left\vert D_{w^{\alpha }}f_{z^{\alpha }}^{\alpha }-D_{w}f_{z}\right\vert
&<&\tau \left( \delta ,\tau \left( {1/\alpha }|\mu \right) \right) \text{ and%
} \\
\left\vert D_{\tilde{w}^{\alpha }}f_{z^{\alpha }}^{\alpha }-D_{\tilde{w}%
}f_{z}\right\vert &<&\tau \left( \delta ,\tau \left( {1/\alpha }|\mu \right)
\right) .
\end{eqnarray*}%
Thus, 
\begin{eqnarray*}
\left\vert D_{w^{\alpha }}f_{z^{\alpha }}^{\alpha }-D_{\tilde{w}^{\alpha
}}f_{z^{\alpha }}^{\alpha }\right\vert &\leq &\left\vert D_{w^{\alpha
}}f_{z^{\alpha }}^{\alpha }-D_{w}f_{z}\right\vert +|D_{w}f_{z}-D_{\tilde{w}%
}f_{z}|+\left\vert D_{\tilde{w}}f_{z}-D_{\tilde{w}^{\alpha }}f_{z^{\alpha
}}^{\alpha }\right\vert \\
&<&|D_{w}f_{z}-D_{\tilde{w}}f_{z}|+\tau \left( \delta ,\tau \left( {1/\alpha 
}|\mu \right) \right) \\
&\leq &\left\vert D_{w}f_{z}-D_{P_{x,y}\left( w\right) }f_{z}\right\vert
+\left\vert D_{P_{x,y}\left( w\right) }f_{z}-D_{\tilde{w}}f_{z}\right\vert
+\tau \left( \delta ,\tau \left( {1/\alpha }|\mu \right) \right) \\
&\leq &\tau \left( \varepsilon ,\delta ,\mu ,\eta ,\tau \left( {1/\alpha }%
|\mu \right) \right)
\end{eqnarray*}%
by the previous corollary and Inequalities \ref{far from p_i} and \ref{mu
move}. Choosing $\varepsilon ,\delta ,\eta ,$ $\mu ,$ and $1/\alpha $ small
enough, we have a contradiction to \ref{big gap}.
\end{proof}

The previous lemma, together with the definitions of $\Phi _{d}^{\alpha },$ $%
\left( {\varphi ^{\eta }}\right) ^{-1}$ and $P_{x^{\alpha },y^{a}}$
establishes the estimates \ref{away from base} and \ref{at base} and hence
the second part of Key Lemma, completing the proof of Theorem \ref{Cross Cap
Stability}, except in dimension $4.$

\section{Recognizing $\mathbb{R}P^{4}$}

To prove Theorem \ref{Cross Cap Stability} in dimension $4,$ we exploit the
following corollary of the fact that $\mathrm{Diff}_{+}\left( S^{3}\right) $
is connected \cite{Cerf}.

\begin{corollary}
\label{Cerf}Let $M$ be a smooth $4$--manifold obtained by smoothly gluing a $%
4$--disk to the boundary of the nontrivial $1$--disk bundle over $\mathbb{R}%
P^{3}.$ Then $M$ is diffeomorphic to $\mathbb{R}P^{4}.$
\end{corollary}

To see that our $M^{\alpha }$s have this structure, we use standard triangle
comparison and argue as we did in the part of Section \ref{Cross Cap Stab}
titled \textquotedblleft Lower Bound on Differential\textquotedblright\ to
conclude

\begin{proposition}
\label{no crtical}For any fixed $\rho _{0}>0,$ $f_{0,d}^{\alpha }$ does not
have critical points on\linebreak\ $M^{\alpha }\setminus \left\{ B\left(
p_{0}^{\alpha },\rho _{0}\right) \cup B\left( A\left( p_{0}^{\alpha }\right)
,\rho _{0}\right) \right\} ,$ and $\nabla f_{0,d}^{\alpha }$ is
gradient-like for $\mathrm{dist}\left( A\left( p_{0}^{\alpha }\right) ,\cdot
\right) $ and $-\mathrm{dist}\left( p_{0}^{\alpha },\cdot \right) ,$
provided $\alpha $ is sufficiently large and $d$ is sufficiently small.
\end{proposition}

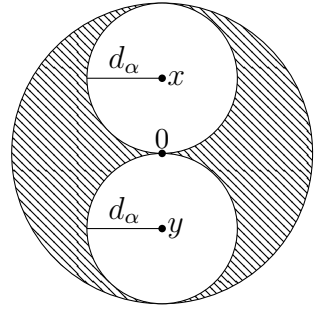
\begin{wrapfigure}{r}{0.5\textwidth}\centering
\vspace{-10pt}
\begin{tikzpicture}[ dot/.style={fill=black,circle,minimum size=1mm}]
\draw[pattern = north west lines] (0,0)circle (2cm);
\draw[fill = white] (0,1cm) circle (1cm);
\draw[fill = white] (0,-1cm) circle (1cm);
\draw (0,0) node[dot]{};
\draw (0,0) node[above]{\small$0$};
\draw (0,1cm) node[dot]{};
\draw (0,1cm) node[right]{$x$};
\draw (0,-1cm) node[dot]{};
\draw (0,-1cm) node[right]{$y$};
\draw (0,1cm) --node[midway,above]{$d_{\alpha}$} (-1cm,1cm);
\draw (0,-1cm) --node[midway,above]{$d_{\alpha}$} (-1cm,-1cm);
\end{tikzpicture}
\caption{The model $\mathcal{D}_{k}^{n}\left(2d_{\alpha
}\right)$.}
\vspace{-20pt}
\end{wrapfigure}

Finally, using Swiss Cheese Volume Comparison (see 1.1 in \cite{GrovPet3})
we will show

\begin{proposition}
\label{swiss cheese}There is a $\rho _{0}>0$ so that $\mathrm{dist}\left(
p_{0}^{\alpha },\cdot \right) $ does not have critical points in $B\left(
p_{0}^{\alpha },\rho_{0}\right),$ provided $\alpha $ is sufficiently large.
\end{proposition}

\begin{proof}
Since $\mathrm{vol}$\thinspace \thinspace $M^{\alpha }\rightarrow \mathrm{vol%
}$\thinspace \thinspace $\mathcal{D}_{k}^{n}\left( r\right) ,$ $\mathrm{vol}$%
\thinspace $B\left( p_{0}^{\alpha },r\right) \rightarrow \mathrm{vol}$%
\thinspace \thinspace $\mathcal{D}_{k}^{n}\left( r\right) .$ Via Swiss
Cheese Volume Comparison (see 1.1 in \cite{GrovPet3}) we shall see that the
presence of a critical point close to $p_{0}^{\alpha }$ contradicts $\mathrm{%
vol}$\thinspace $B\left( p_{0}^{\alpha },r\right) \rightarrow \mathrm{vol}$%
\thinspace \thinspace $\mathcal{D}_{k}^{n}\left( r\right) .$ Suppose $%
q_{\alpha }$ is critical for $\mathrm{dist}\left( p_{0}^{\alpha },\cdot
\right) ,$ and $\mathrm{dist}\left( p_{0}^{\alpha },q_{\alpha }\right)
=d_{\alpha }\rightarrow 0$. Let $x,y$ be points in $\partial \mathcal{D}%
_{k}^{n}\left( d_{\alpha }\right) $ at maximal distance. By Swiss Cheese
Comparison and 1.4 in \cite{GrovPet3}, 
\begin{eqnarray*}
\mathrm{vol}\text{\thinspace }\left( B\left( q_{\alpha },2d_{\alpha }\right)
\setminus B\left( p_{0}^{\alpha },d_{\alpha }\right) \right) &\leq &\mathrm{%
vol}\text{\thinspace \thinspace }\left( \mathcal{D}_{k}^{n}\left( 2d_{\alpha
}\right) \setminus \left\{ B\left( x,d_{\alpha }\right) \cup B\left(
y,d_{\alpha }\right) \right\} \right) \\
&=&\mathrm{vol}\text{\thinspace \thinspace }\left( \mathcal{D}_{k}^{n}\left(
2d_{\alpha }\right) \right) -2\mathrm{vol}\text{\thinspace \thinspace }%
\left( \mathcal{D}_{k}^{n}\left( d_{\alpha }\right) \right) .
\end{eqnarray*}%
Since 
\begin{equation*}
\mathrm{vol}\text{\thinspace }B\left( p_{0}^{\alpha },d_{\alpha }\right)
\leq \mathrm{vol}\text{\thinspace \thinspace }\mathcal{D}_{k}^{n}\left(
d_{\alpha }\right) ,
\end{equation*}%
we conclude%
\begin{eqnarray*}
\mathrm{vol}\text{\thinspace }\left( B\left( q_{\alpha },2d_{\alpha }\right)
\right) &\leq &\mathrm{vol}\text{\thinspace \thinspace }\left( \mathcal{D}%
_{k}^{n}\left( 2d_{\alpha }\right) \right) -\mathrm{vol}\text{\thinspace
\thinspace }\left( \mathcal{D}_{k}^{n}\left( d_{\alpha }\right) \right) \\
&<&\kappa \cdot \mathrm{vol}\text{\thinspace \thinspace }\mathcal{D}%
_{k}^{n}\left( 2d_{\alpha }\right)
\end{eqnarray*}%
for some $\kappa \in \left( 0,1\right) .$ By relative volume comparison for $%
\rho \geq 2d_{\alpha },$ 
\begin{equation*}
\kappa >\frac{\mathrm{vol}\text{\thinspace }B\left( q_{\alpha },2d_{\alpha
}\right) }{\mathrm{vol}\text{\thinspace \thinspace }\mathcal{D}%
_{k}^{n}\left( 2d_{\alpha }\right) }\geq \frac{\mathrm{vol}\text{\thinspace }%
B\left( q_{\alpha },\rho \right) }{\mathrm{vol}\text{\thinspace \thinspace }%
\mathcal{D}_{k}^{n}\left( \rho \right) }
\end{equation*}%
or 
\begin{equation*}
\kappa \cdot \mathrm{vol}\text{\thinspace \thinspace }\mathcal{D}%
_{k}^{n}\left( \rho \right) >\mathrm{vol}\text{\thinspace }B\left( q_{\alpha
},\rho \right) .
\end{equation*}%
Since 
\begin{eqnarray*}
B\left( p_{0}^{\alpha },r\right) &\subset &B\left( q_{\alpha },r+d_{\alpha
}\right) , \\
\mathrm{vol}\text{\thinspace }B\left( p_{0}^{\alpha },r\right) &<&\kappa
\cdot \mathrm{vol}\text{\thinspace \thinspace }\mathcal{D}_{k}^{n}\left(
r+d_{\alpha }\right) .
\end{eqnarray*}%
Letting $d_{\alpha }\rightarrow 0,$ we conclude that 
\begin{equation*}
\mathrm{vol}\text{\thinspace }B\left( p_{0}^{\alpha },r\right) <\kappa \cdot 
\mathrm{vol}\text{\thinspace \thinspace }\mathcal{D}_{k}^{n}\left( r\right) ,
\end{equation*}%
a contradiction.
\end{proof}

An identical argument shows

\begin{proposition}
There is a $\rho _{0}>0$ so that $\mathrm{dist}\left( A\left( p_{0}^{\alpha
}\right) ,\cdot \right) $ does not have critical points in $B\left( A\left(
p_{0}^{\alpha }\right) ,\rho \right) ,$ provided $\alpha $ is sufficiently
large.
\end{proposition}

Combining the previous three propositions, we see that $(f_{0,d}^{\alpha
})^{-1}\left( 0\right) $ is diffeomorphic to $S^{3}.$ By Geometrization, $%
(f_{0,d}^{\alpha })^{-1}\left( 0\right) /\left\{ \mathrm{id},A\right\} $ is
diffeomorphic to $\mathbb{R}P^{3}.$ If $\rho _{0}$ is as in Proposition \ref%
{no crtical}, it follows that $(f_{0,d}^{\alpha })^{-1}(\left[ -\rho
_{0},\rho _{0}\right] )/\left\{ \mathrm{id},A\right\} $ is the nontrivial $1$%
--disk bundle over $\mathbb{R}P^{3}.$ $\tilde{M}^{\alpha }\setminus
(f_{0,d}^{\alpha })^{-1}(\left[ -\rho _{0},\rho _{0}\right] )$ consists of
two smooth $4$--disks that get interchanged by $A.$ Thus $M^{\alpha }$ has
the structure of Corollary \ref{Cerf} and is hence diffeomorphic to $\mathbb{%
R}P^{4}.$

\begin{remark}
The proof of Perelman's Parameterized Stability Theorem \cite{Kap} can
substitute for Geometrization to allow us to conclude that $f^{-1}\left(
0\right) /\left\{ \mathrm{id},A\right\} $ is homeomorphic and therefore
diffeomorphic to $\mathbb{R}P^{3}.$ The need to cite the proof rather than
the theorem stems from the fact that the definition of admissible functions
in \cite{Kap} excludes $f_{0,d}^{\alpha }$ . It is straightforward (but
tedious) to see that the proof goes through for an abstract class that
includes $f_{0,d}^{\alpha }.$

The fact that $\mathbb{R}P^{4}$ admits exotic differential structures can be
seen by combining \cite{HKT} with either \cite{CS} or \cite{FintSt}.
\end{remark}


\medskip

\begin{thebibliography}{99}
\bibitem{BGP} Y. Burago, M. Gromov, G. Perelman, \emph{A.D. Alexandrov
spaces with curvatures bounded from below}, I, Uspechi Mat. Nauk. \textbf{47}
(1992), 3--51.

\bibitem{Bar} D. Barden, \emph{The structure of manifolds,} Ph.D. Thesis,
Cambridge University, Cambridge, England.

\bibitem{Cerf} J. Cerf, \emph{La stratification naturelle des espaces de
fonctions diff\'{e}rntiables r\'{e}elles et le th\'{e}or\`{e}me de la
pseudo-isotopie}, Publ. Math. I.H.E.S. \textbf{39} (1970), 5-173.

\bibitem{CS} S. E. Cappell and J. L. Shaneson, \emph{Some new four-manifolds.%
} Ann. of Math. \textbf{104} (1976), 61-72.

\bibitem{Cheeg1} J. Cheeger, \emph{Comparison and finiteness theorems for
Riemannian manifolds,} Thesis, Princeton University, 1967.

\bibitem{Cheeg2} J. Cheeger, \emph{Finiteness theorems for Riemannian
manifolds}, Amer. J. Math. 92 (1970) 61-74.

\bibitem{FintSt} R. Fintushel and R. Stern, \emph{An exotic free involution
on }$S^{4},$ Ann. of Math. \textbf{113 }(1981), 357--365.

\bibitem{Fuk} K. Fukaya.\emph{\ Theory of convergence for Riemannian
orbifolds}. Japan. J. Math., \textbf{12} (1986), 121--160.

\bibitem{FukYam} K. Fukaya and T. Yamaguchi, \emph{Isometry groups of
singular spaces}. Math. Z. \textbf{216} (1994), 31--44.

\bibitem{GrWu} R. Greene and H. Wu, \emph{Integrals of subharmonic functions
on manifolds of nonnegative curvature, }Inventiones Math. \textbf{27}(1974)
265-298.

\bibitem{GrovPet1} K. Grove and P. Petersen, \emph{Bounding homotopy types
by geometry, } Ann. of Math. \textbf{128} (1988), 195-206.

\bibitem{GrovPet2} K. Grove and P. Petersen, \emph{Manifolds near the
boundary of existence, } J. Diff. Geom. \textbf{33} (1991), 379-394.

\bibitem{GrovPet3} K. Grove and P. Petersen, \emph{Volume comparison \`{a}
la Alexandrov}, Acta. Math. \textbf{169} (1992), 131-151.

\bibitem{GrovShio} K. Grove and K. Shiohama, \emph{A generalized sphere
theorem, }Ann. of Math. \textbf{106} (1977), 201-211.

\bibitem{GrovWilh1} K. Grove and F. Wilhelm,\emph{\ Hard and soft packing
radius theorems.} Ann. of Math. \textbf{142} (1995), 213--237.

\bibitem{GrovWilh2} K. Grove and F. Wilhelm, \emph{Metric constraints on
exotic spheres via Alexandrov geometry}. J. Reine Angew. Math. \textbf{487}
(1997), 201--217.

\bibitem{Kap} V. Kapovitch,\emph{\ Perelman's stability theorem.} Surveys in
differential geometry. \textbf{11} (2007), 103-136.

\bibitem{HKT} I. Hambleton, M. Kreck, and Teichner,\emph{\ Non-orientable
4-manifolds with fundamental group of order 2.} Trans. Amer. Math. Soc. 
\textbf{344} (1994), 649-665.

\bibitem{Hig} G. Higman, \emph{The units of group-rings,} Proc. London Math.
Soc. \textbf{46} (1940), 231--248.

\bibitem{LB} W. LaBach, \emph{On diffeomorphisms of the }$n$\emph{--disk, }%
Proc. Japan Acad. \textbf{43 }(1967), 448-450.

\bibitem{KM} M. Kervaire and J. Milnor, \emph{Groups of homotopy spheres: I}%
, Ann. of Math. \textbf{77} (1963), 504-537.

\bibitem{KMS} Kuwae, K., Machigashira, Y., and Shioya T., \emph{Sobolev
spaces, Laplacian, and heat kernel on Alexandrov spaces,} Math. Z. 238
(2001), no. 2, 269--316.

\bibitem{Maz} B. Mazur, \emph{Relative neighborhoods and the theorems of
Smale}, Ann. of Math \textbf{77}, (1963), 232-249.

\bibitem{Mil1} J. Milnor, \emph{Lectures on the H-Cobordism Theorem},
Princeton University Press (1965).

\bibitem{Mil2} J. Milnor, \emph{Whitehead torsion} Bull. Amer. Math. Soc. 
\textbf{72} (1966), 358--426.

\bibitem{OSY} Y. Otsu, K. Shiohama and T. Yamaguchi, \emph{A new version of
differentiable sphere theorem.} Invent. Math. \textbf{98} (1989), 219--228.

\bibitem{OS} Y. Otsu, T. Shioya, \emph{The Riemannian Structure of
Alexandrov Spaces} J. Differential Geometry \textbf{39} (1994), 629--658.

\bibitem{NiRong} N. Li, X. Rong, \emph{Relative Volume Rigidity in
Alexandrov Geometry, } Pacific J. Math. \textbf{259} (2012), no. 2, 387--420.

\bibitem{Perel} G. Perelman, \emph{Alexandrov spaces with curvature bounded
from below II, }preprint 1991.

\bibitem{Pet} A. Petrunin, \emph{Semiconcave functions in Alexandrov's
Geometry} Surv. in Diff. \textbf{11} (2007), 137-201.

\bibitem{ProSillWilh} C. Pro, M. Sill., and F. Wilhelm, \emph{The
diffeomorphism type of manifolds with almost maximal volume, }preprint.

\bibitem{Rud} W. Rudin, \emph{Principles of mathematical analysis.} Third
edition. International Series in Pure and Applied Mathematics. McGraw-Hill
Book Co., New York-Auckland-D\"{u}sseldorf, (1976)

\bibitem{SY} K. Shiohama, T. Yamaguchi, \emph{Positively curved manifolds
with restricted diameters}, Perspectives in Math. \textbf{8} (1989), 345-350.

\bibitem{SorWei} C. Sormani, G. Wei, \emph{\ Universal covers for Hausdorff
limits of noncompact spaces} Trans. Amer. Math. Soc. \textbf{356} (2004),
1233 -- 1270.

\bibitem{Stal} J. Stallings, \emph{Projective class groups and Whitehead
groups,} (mimeographed) Rice University, Houston, Texas

\bibitem{Steen} N. Steenrod, \emph{Topology of Fibre Bundles, }Princeton U.
Press, 1951.

\bibitem{Wilh} F. Wilhelm, \emph{Collapsing to almost Riemannian spaces.}
Indiana Univ. Math. J. \textbf{41} (1992), 1119--1142.

\bibitem{Yam1} T. Yamaguchi, \emph{Collapsing and pinching under a lower
curvature bound}. Ann. of Math. \textbf{133} (1991), 317--357.

\bibitem{Yam2} T. Yamaguchi, \emph{A convergence theorem in the geometry of
Alexandrov spaces}. Actes de la Table Ronde de G\'{e}om\'{e}trie Diff\'{e}%
rentielle. (1992), 601--642.
\end{thebibliography}
\end{document}